\documentclass[11pt,dvips,twoside,letterpaper]{article}
\usepackage[usenames]{color}
\usepackage{pslatex}
\usepackage{fancyhdr}
\usepackage{graphicx}
\usepackage{geometry}
\usepackage{amsmath}
\usepackage{amssymb}
\usepackage{amsfonts}
\usepackage{amsthm,amscd}

\def\figurename{Figure}
\makeatletter
\renewcommand{\fnum@figure}[1]{\figurename~\thefigure.}
\makeatother
\def\tablename{Table}
\makeatletter
\renewcommand{\fnum@table}[1]{\tablename~\thetable.}
\makeatother
\newtheorem{theorem}{Theorem}[section]
\newtheorem{lemma}[theorem]{Lemma}
\newtheorem{corollary}[theorem]{Corollary}

\theoremstyle{definition}
\newtheorem{definition}[theorem]{Definition}
\newtheorem{example}[theorem]{Example}

\theoremstyle{remark}
\newtheorem{remark}[theorem]{Remark}
\numberwithin{equation}{section}

\input{tcilatex}
\begin{document}

\title{\bfseries\scshape{  A More Accurate Half-Discrete Hardy-Hilbert-Type Inequality with the
Best Possible Constant Factor Related to the Extended Riemann-Zeta Function }}
\author{Michael Th. Rassias$^{1\ast }$ and Bicheng Yang$^{2}$ \and 1*.
Institute of Mathematics, University of Zurich, CH - 8057, \and Zurich, Switzerland
\and \& Institute for Advanced Study, Program in Interdisciplinary Studies,\and
1 Einstein Dr, Princeton, NJ 08540, USA.
\and 2. Department of Mathematics, Guangdong University of \and Education,
Guangzhou, Guangdong 510303, P. R. China \and E-mail: 1*. Corresponding
author: michael.rassias@math.ethz.ch \and 2. bcyang@gdei.edu.cn\,\,\,\,
bcyang818@163.com}
%EndAName
\date{}
\maketitle
\thispagestyle{empty} \setcounter{page}{1}

\renewcommand{\headrulewidth}{0pt}
\begin{abstract}
   By the method of weight coefficients, techniques of real analysis and
Hermite-Hadamard's inequality, a half-discrete Hardy-Hilbert-type inequality
related to the kernel of the hyperbolic cosecant function with the best possible
constant factor expressed in terms of the extended Riemann-zeta function is proved.
The more accurate equivalent forms, the operator expressions with the norm,
the reverses and some particular cases are also considered. \newline

\textbf{Key words:}     Hardy-Hilbert-type inequality; extended Riemann-zeta function; Hurwitz zeta function; Gamma function; weight function;
equivalent form; operator

%\textbf{2000 Mathematics Subject Classification.}  65B10, \,\,26D15, \,\,47A07\newline
\textbf{\ }
\end{abstract}

\section{Introduction}

If $p>1,\frac{1}{p}+\frac{1}{q}=1,f(x),g(y)\geq 0,f\in L^{p}(\mathbf{R}%
_{+}),g\in L^{q}(\mathbf{R}_{+}),$
$$||f||_{p}=\left(\int_{0}^{\infty
}f^{p}(x)dx\right)^{\frac{1}{p}}>0,$$ and $||g||_{q}>0,$ then we have the following
Hardy-Hilbert's integral inequality (cf. \cite{HLP}):
\begin{equation}
\int_{0}^{\infty }\int_{0}^{\infty }\frac{f(x)g(y)}{x+y}dxdy<\frac{\pi }{%
\sin (\pi /p)}||f||_{p}||g||_{q},  \label{1.1}
\end{equation}%
where, the constant factor $\frac{\pi }{\sin (\pi /p)}$ is the best
possible. Assuming that 
$$a_{m},b_{n}\geq 0,a=\{a_{m}\}_{m=1}^{\infty }\in
l^{p}, b=\{b_{n}\}_{n=1}^{\infty }\in l^{q}, ||a||_{p}=\left(\sum_{m=1}^{%
\infty }a_{m}^{p}\right)^{\frac{1}{p}}>0,||b||_{q}>0,$$ we have the following
discrete analogue of (\ref{1.1})   with the same best constant $\frac{\pi }{%
\sin (\pi /p)}$ (cf. \cite{HLP}):%
\begin{equation}
\sum_{m=1}^{\infty }\sum_{n=1}^{\infty }\frac{a_{m}b_{n}}{m+n}<\frac{\pi }{%
\sin (\pi /p)}||a||_{p}||b||_{q}.  \label{1.2}
\end{equation}%
Inequalities (\ref{1.1}) and (\ref{1.2}) are important in Mathematical Analysis and its
applications (cf. \cite{HLP}, \cite{MPF}, \cite{Y1}, \cite{Y2}, \cite{Y4}).

Suppose that $\mu _{i},\upsilon _{j}>0\ (i,j\in \mathbf{N=\{}1,2,\cdots
\mathbf{\}}),$%
\begin{equation}
U_{m}:=\sum_{i=1}^{m}\mu _{i},V_{n}:=\sum_{j=1}^{n}\nu _{j}\ \ (m,n\in \mathbf{N}%
).  \label{1.3}
\end{equation}%
Then we have the following inequality (cf. \cite{HLP}, Theorem 321,
replacing $\mu _{m}^{1/q}a_{m}$ and $\upsilon _{n}^{1/p}b_{n}$ by $a_{m}$
and $b_{n}$) :
\begin{equation}
\sum_{m=1}^{\infty }\sum_{n=1}^{\infty }\frac{a_{m}b_{n}}{U_{m}+V_{n}}<\frac{%
\pi }{\sin (\frac{\pi }{p})}\left( \sum_{m=1}^{\infty }\frac{a_{m}^{p}}{\mu
_{m}^{p-1}}\right) ^{\frac{1}{p}}\left( \sum_{n=1}^{\infty }\frac{b_{n}^{q}}{%
\nu _{n}^{q-1}}\right) ^{\frac{1}{q}}.  \label{1.4}
\end{equation}%
For $\mu _{i}=\upsilon _{j}=1\ (i,j\in \mathbf{N}),$ inequality (\ref{1.4})
reduces to (\ref{1.2}). We call (\ref{1.4}) Hardy-Hilbert-type inequality.

\textbf{Note.} The authors of \cite{HLP} did not prove that (\ref{1.4}) is
valid with the best possible constant factor.

In 1998, by introducing an independent parameter $\lambda \in (0,1]$, Yang
\cite{Y3} obtained an extension of (\ref{1.1}) with the kernel $\frac{1}{%
(x+y)^{\lambda }}$ for $p=q=2$. Refining the method applied in \cite{Y3}, Yang \cite%
{Y4} provided extensions of (\ref{1.1}) and (\ref{1.2}) as follows:

Assuming that $\lambda _{1},\lambda _{2}\in \mathbf{R},\lambda _{1}+\lambda
_{2}=\lambda ,k_{\lambda }(x,y)$ is a non-negative homogeneous function of
degree $-\lambda ,$ with $$k(\lambda _{1})=\int_{0}^{\infty }k_{\lambda
}(t,1)t^{\lambda _{1}-1}dt\in \mathbf{R}_{+},$$ 
$$\phi (x)=x^{p(1-\lambda
_{1})-1},\ \ \psi (x)=x^{q(1-\lambda _{2})-1},f(x),g(y)\geq 0,$$
\begin{equation*}
f\in L_{p,\phi }(\mathbf{R}_{+})=\left\{ f;||f||_{p,\phi
}:=\{\int_{0}^{\infty }\phi (x)|f(x)|^{p}dx\}^{\frac{1}{p}}<\infty \right\} ,
\end{equation*}%
where $g\in L_{q,\psi }(\mathbf{R}_{+}),||f||_{p,\phi },||g||_{q,\psi }>0,$ we
have
\begin{equation}
\int_{0}^{\infty }\int_{0}^{\infty }k_{\lambda }(x,y)f(x)g(y)dxdy<k(\lambda
_{1})||f||_{p,\phi }||g||_{q,\psi },  \label{1.5}
\end{equation}%
where, the constant factor $k(\lambda _{1})$ is the best possible. Moreover,
if $k_{\lambda }(x,y)$ keeps finite and $k_{\lambda }(x,y)x^{\lambda
_{1}-1}(k_{\lambda }(x,y)y^{\lambda _{2}-1})$ is decreasing with respect to $%
x>0\ (y>0),$ then for $a_{m,}b_{n}\geq 0,$%
\begin{equation*}
a\in l_{p,\phi }=\left\{ a;||a||_{p,\phi }:=\left(\sum_{n=1}^{\infty }\phi
(n)|a_{n}|^{p}\right)^{\frac{1}{p}}<\infty \right\} ,
\end{equation*}%
$b=\{b_{n}\}_{n=1}^{\infty }\in l_{q,\psi },$ $||a||_{p,\phi },||b||_{q,\psi
}>0,$ we have
\begin{equation}
\sum_{m=1}^{\infty }\sum_{n=1}^{\infty }k_{\lambda
}(m,n)a_{m}b_{n}<k(\lambda _{1})||a||_{p,\phi }||b||_{q,\psi },  \label{1.6}
\end{equation}%
where, the constant factor $k(\lambda _{1})$ is still the best possible.

For $0<\lambda _{1},\lambda _{2}\leq 1,\lambda _{1}+\lambda _{2}=\lambda ,$
we set%
\begin{equation*}
k_{\lambda }(x,y)=\frac{1}{(x+y)^{\lambda }}\ \ ((x,y)\in \mathbf{R}_{+}^{2}).
\end{equation*}%
Then by (\ref{1.6}), we have
\begin{equation}
\sum_{m=1}^{\infty }\sum_{n=1}^{\infty }\frac{a_{m}b_{n}}{(m+n)^{\lambda }}%
<B(\lambda _{1},\lambda _{2})||a||_{p,\phi }||b||_{q,\psi },  \label{1.7}
\end{equation}%
where, the constant $B(\lambda _{1},\lambda _{2})$ is the best possible, and%
\begin{equation*}
B\left( u,v\right) =\int_{0}^{\infty }\frac{1}{(1+t)^{u+v}}%
t^{u-1}dt^{{}}\ \ (u,v>0)
\end{equation*}%
is the beta function. Clearly, for $\lambda =1,\lambda _{1}=\frac{1}{q}%
,\lambda _{2}=\frac{1}{p},$ inequality (\ref{1.7}) reduces to (\ref{1.2}).

In 2015, by adding some conditions, Yang \cite{Y4a} extended (%
\ref{1.7}) and (\ref{1.4}) as follows:
\begin{eqnarray}
&&\sum_{m=1}^{\infty }\sum_{n=1}^{\infty }\frac{a_{m}b_{n}}{%
(U_{m}+V_{n})^{\lambda }}  \notag \\
&<&B(\lambda _{1},\lambda _{2})\left[ \sum_{m=1}^{\infty }\frac{%
U_{m}^{p(1-\lambda _{1})-1}a_{m}^{p}}{\mu _{m}^{p-1}}\right] ^{\frac{1}{p}%
}\left[ \sum_{n=1}^{\infty }\frac{V_{n}^{q(1-\lambda _{2})-1}b_{n}^{q}}{\nu
_{n}^{q-1}}\right] ^{\frac{1}{q}},  \label{1.8}
\end{eqnarray}%
where, the constant $B(\lambda _{1},\lambda _{2})$ is still the best
possible.

Some other results including multidimensional Hilbert-type inequalities are
provided in \cite{YMP}-\cite{Y6}.

Related to the topic of half-discrete Hilbert-type inequalities with the
non-homogeneous kernels, Hardy et al. provided a few results in Theorem 351
of \cite{HLP}. But they did not prove that the constant factors are the
best possible. However, Yang \cite{Y5} established a result with the kernel $\frac{1%
}{(1+nx)^{\lambda }}$ by introducing a variable and proved that the constant
factor is the best possible. In 2011 Yang \cite{122} proved the following
half-discrete Hardy-Hilbert's inequality with the best possible constant
factor $B\left( \lambda _{1},\lambda _{2}\right) $:%
\begin{equation}
\int_{0}^{\infty }f\left( x\right) \left[ \sum_{n=1}^{\infty }\frac{a_{n}}{%
\left( x+n\right) ^{\lambda }}\right] dx<B\left( \lambda _{1},\lambda
_{2}\right) ||f||_{p,\phi }||a||_{q,\psi },  \label{1.9}
\end{equation}%
where, $\lambda _{1}>0$, $0<\lambda _{2}\leq 1$, $\lambda _{1}+\lambda
_{2}=\lambda .$ Zhong et \emph{al} (\cite{140}--\cite{144}) investigated
several half-discrete Hilbert-type inequalities with particular kernels.
Applying the method of weight functions, a half-discrete Hilbert-type
inequality with a general homogeneous kernel of degree $-\lambda \in \mathbf{%
R}$ and a best constant factor $k\left( \lambda _{1}\right) $ is obtained as
follows:%
\begin{equation}
\int_{0}^{\infty }f(x)\sum_{n=1}^{\infty }k_{\lambda }(x,n)a_{n}dx<k(\lambda
_{1})||f||_{p,\phi }||a||_{q,\psi },  \label{1.10}
\end{equation}%
which is an extension of (\ref{1.9}) (cf. \cite{125}). At the same time, a
half-discrete Hilbert-type inequality with a general non-homogeneous kernel
and a best constant factor is given by Yang \cite{YB1}. In 2012-2014, Yang
et al. published three books \cite{YB8}, \cite{YB9} and \cite{YB7} extensively 
presenting the framework of half-discrete Hilbert-type inequalities.

In this paper, by the method of weight coefficients, techniques of real analysis and
Hermite-Hadamard's inequality, a half-discrete
Hardy-Hilbert-type inequality related to the kernel of the hyperbolic cosecant
function with a best possible constant factor expressed by the extended
Riemann-zeta function is proved, which is an extension of (\ref{1.10}) for $%
\lambda =0$ in the following particular kernel:
\begin{equation*}
k_{0}(x,n)=\frac{\csc h(\rho (\frac{n}{x})^{\gamma })}{e^{\alpha (\frac{n}{x}%
)^{\gamma }}}(\rho >\max \{0,-\alpha \},0<\gamma <1).
\end{equation*}
 Furthermore, the more accurate
equivalent forms, the operator expressions with the norm, the reverses and
some particular cases are also considered.

\section{Some Lemmas}
In the sequel, we shall assume that $\nu _{n}>0\ (n\in \mathbf{N}%
),\{\upsilon _{n}\}_{n=1}^{\infty }$ is decreasing, $V_{n}=\sum_{j=1}^{n}\nu
_{j}$, $\mu (t)$ is a positive continuous function in $\mathbf{R}%
_{+}=(0,\infty )$,
\begin{equation*}
U(0):=0;\ \ U(x):=\int_{0}^{x}\mu (t)dt<\infty ^{{}}(x\in (0,\infty )),
\end{equation*}%
$$\nu (t):=\nu _{n},\ t\in (n-1,n]\ (n\in \mathbf{N}),$$ 
and
\begin{equation*}
V(0):=0;\ \ V(y):=\int_{0}^{y}\nu (t)dt^{{}}(y\in (0,\infty )),
\end{equation*}%
$p\neq 0,1,$ $\frac{1}{p}+\frac{1}{q}=1,\delta \in \{-1,1\},\beta \leq \frac{%
\nu _{1}}{2},$ $f(x),a_{n}\geq 0\ (x\in \mathbf{R}_{+},n\in \mathbf{N}),$ $$
||f||_{p,\Phi _{\delta }}=(\int_{0}^{\infty }\Phi _{\delta }(x)f^{p}(x)dx)^{%
\frac{1}{p}},$$ 
$$||a||_{q,\Psi }=(\sum_{n=1}^{\infty }\Psi _{\beta
}(n)b_{n}^{q})^{\frac{1}{q}},$$ where,
\begin{eqnarray*}
\Phi _{\delta }(x) &:&=\frac{U^{p(1-\delta \sigma )-1}(x)}{\mu ^{p-1}(x)}%
(x\in \mathbf{R}_{+}), \\
\Psi _{\beta }(n) &:&=\frac{(V_{n}-\beta )^{q(1-\sigma )-1}}{\nu _{n+1}^{q-1}%
}(n\in \mathbf{N}).
\end{eqnarray*}

\begin{lemma}
  If $a\in \mathbf{R},$ $f(x)$ is continuous in $[a-\frac{1}{%
2},a+\frac{1}{2}],$ $f^{\prime }(x)$ is strictly increasing in $(a-\frac{1}{2}%
,a)$ and $(a,a+\frac{1}{2})$ respectively, as well as%
\begin{equation*}
\lim_{x\rightarrow a-}f^{\prime }(x)=f^{\prime }(a-0)\leq f^{\prime
}(a+0)=\lim_{x\rightarrow a+}f^{\prime }(x),
\end{equation*}%
then $f(x)$ is strictly convex in $[a-\frac{1}{2},a+\frac{1}{2}],$ and we
have the following Hermite-Hadamard's inequality (cf. \cite{K1}):%
\begin{equation}
f(a)<\int_{a-\frac{1}{2}}^{a+\frac{1}{2}}f(x)dx.  \label{2.1}
\end{equation}
\end{lemma}

\begin{proof}
 Since $f^{\prime }(a-0)\ (\leq f^{\prime }(a+0))$ is finite, we
define a function $g(x)$ as follows:
\begin{equation*}
g(x):=f^{\prime }(a-0)(x-a)+f(a),x\in \lbrack a-\frac{1}{2},a+\frac{1}{2}].
\end{equation*}%
In view of $f^{\prime }(x)$ being strictly increasing in $(a-\frac{1}{2},a),$
then for $x\in (a-\frac{1}{2},a),$ 
$$(f(x)-g(x))^{\prime } =f^{\prime
}(x)-f^{\prime }(a-0)<0.$$ Since $f(a)-g(a)=0,$ it follows that $%
f(x)-g(x)>0,\ x\in (a-\frac{1}{2},a).$ Similarly, we can obtain $%
f(x)-g(x)>0,\ x\in (a,a+\frac{1}{2}).$ Hence, $f(x)$ is strictly convex in $[a-%
\frac{1}{2},a+\frac{1}{2}],$ and therefore%
\begin{equation*}
\int_{a-\frac{1}{2}}^{a+\frac{1}{2}}f(x)dx>\int_{a-\frac{1}{2}}^{a+\frac{1}{2%
}}g(x)dx=f(a),
\end{equation*}%
namely, (\ref{2.1}) follows.
\end{proof}

\begin{example}
  For $\rho >\max \{0,-\alpha \},0<\gamma <\sigma \leq 1,$
$$\csc h(u)=\frac{2}{e^{u}-e^{-u}}\ (u>0)$$ is called hyperbolic cosecant function
(cf. \cite{ZYQ}), \ we set%
 \begin{equation*}
h(t)=\frac{\csc h(\rho t^{\gamma })}{e^{\alpha t^{\gamma }}}=\frac{2}{%
e^{(\alpha +\rho )t^{\gamma }}(1-e^{-2\rho t^{\gamma }})}\ \ (t\in \mathbf{R}%
_{+}).
\end{equation*}
\end{example}

(i) Setting $u=\rho t^{\gamma },$ we find%
\begin{eqnarray*}
k(\sigma ) &:=&\int_{0}^{\infty }\frac{\csc h(\rho t^{\gamma })}{e^{\alpha
t^{\gamma }}}t^{\sigma -1}dt \\
&=&\frac{1}{\gamma \rho ^{\sigma /\gamma }}\int_{0}^{\infty }\frac{\csc h(u)%
}{e^{\frac{\alpha }{\rho }u}}u^{\frac{\sigma }{\gamma }-1}du \\
&=&\frac{2}{\gamma \rho ^{\sigma /\gamma }}\int_{0}^{\infty }\frac{e^{-\frac{%
\alpha }{\rho }u}u^{\frac{\sigma }{\gamma }-1}}{e^{u}-e^{-u}}du \\
&=&\frac{2}{\gamma \rho ^{\sigma /\gamma }}\int_{0}^{\infty }\frac{e^{-(%
\frac{\alpha }{\rho }+1)u}u^{\frac{\sigma }{\gamma }-1}}{1-e^{-2u}}du \\
&=&\frac{2}{\gamma \rho ^{\sigma /\gamma }}\int_{0}^{\infty
}\sum_{k=0}^{\infty }e^{-(2k+\frac{\alpha }{\rho }+1)u}u^{\frac{\sigma }{%
\gamma }-1}du.
\end{eqnarray*}%
By Lebesgue's term by term theorem (cf. \cite{ZYQ}), setting $v=(2k+\frac{%
\alpha }{\rho }+1)u,$ we have%
\begin{eqnarray}
k(\sigma ) &=&\int_{0}^{\infty }\frac{\csc h(\rho t^{\gamma })}{e^{\alpha
t^{\gamma }}}t^{\sigma -1}dt  \notag \\
&=&\frac{2}{\gamma \rho ^{\sigma /\gamma }}\sum_{k=0}^{\infty
}\int_{0}^{\infty }e^{-(2k+\frac{\alpha }{\rho }+1)u}u^{\frac{\sigma }{%
\gamma }-1}du  \notag \\
&=&\frac{2}{\gamma \rho ^{\sigma /\gamma }}\sum_{k=0}^{\infty }\frac{1}{(2k+%
\frac{\alpha }{\rho }+1)^{\sigma /\gamma }}\int_{0}^{\infty }e^{-v}v^{\frac{%
\sigma }{\gamma }-1}dv  \notag \\
&=&\frac{2\Gamma (\frac{\sigma }{\gamma })}{\gamma (2\rho )^{\sigma /\gamma }%
}\sum_{k=0}^{\infty }\frac{1}{(k+\frac{\alpha +\rho }{2\rho })^{\sigma
/\gamma }}  \notag \\
&=&\frac{2\Gamma (\frac{\sigma }{\gamma })}{\gamma (2\rho )^{\sigma /\gamma }%
}\zeta (\frac{\sigma }{\gamma },\frac{\alpha +\rho }{2\rho })\in \mathbf{R}%
_{+},  \label{2.2}
\end{eqnarray}%
where 
$$\zeta (s,a):=\sum_{k=0}^{\infty }\frac{1}{(k+a)^{s}}\ \ (\textit{Re}%
(s)>1,a>0)$$ 
is called the extended Riemann-zeta function (also known as the Hurwitz zeta function)\footnote{Clearly $\zeta (s,1)=\zeta
(s)$, where $\zeta(s)$ is the Riemann-zeta function.}, and%
\begin{equation*}
\Gamma (y):=\int_{0}^{\infty }e^{-v}v^{y-1}dv^{{}}\ (y>0)
\end{equation*}%
is called Gamma function (cf. \cite{Y4b}).

\noindent In particular, for $\alpha =\rho ,$ we have $$h(t)=\frac{\csc h(\rho
t^{\gamma })}{e^{\rho t^{\gamma }}}\ \ \text{and}\ \ k(\sigma )=k_{1}(\sigma ):=\frac{%
2\Gamma (\frac{\sigma }{\gamma })}{\gamma (2\rho )^{\sigma /\gamma }}\zeta (%
\frac{\sigma }{\gamma }).$$ 
In this case, for $\gamma =\frac{\sigma }{2},$ we have 
$$%
h(t)=\frac{\csc h(\rho t^{\sigma /2})}{e^{\rho t^{\sigma /2}}}\ \ \text{and}\ \ 
k(\sigma )=\frac{\pi ^{2}}{6\sigma \rho ^{2}}.$$

\noindent (ii) We obtain for $u>0$ that 
$$\frac{1}{1-e^{-2u}}>0,\ \ (\frac{1}{1-e^{-2u}}%
)^{\prime }=-\frac{2e^{-2u}}{(1-e^{-2u})^{2}}<0,$$ 
and%
\begin{equation*}
(\frac{1}{1-e^{-2u}})^{\prime \prime }=\frac{4e^{-2u}}{(1-e^{-2u})^{2}}+%
\frac{8e^{-4u}}{(1-e^{-2u})^{3}}>0.
\end{equation*}

(iii) If $g(u)>0,g^{\prime }(u)<0,g^{\prime \prime }(u)>0,$ then for $%
0<\gamma \leq 1,$ we find that $g(\rho t^{\gamma })>0,$ $\frac{d}{dt}g(\rho
t^{\gamma })=\rho \gamma t^{\gamma -1}g^{\prime }(\rho t^{\gamma })<0,$ and
\begin{equation*}
\frac{d^{2}}{dt^{2}}g(\rho t^{\gamma })=\rho \gamma (\gamma -1)t^{\gamma
-2}g^{\prime }(\rho t^{\gamma })+(\rho \gamma t^{\gamma -1})^{2}g^{\prime
\prime }(\rho t^{\gamma })>0.
\end{equation*}%
Then we find that for $y\in (n-\frac{1}{2},n),$
$$g(V(y)-\beta )>0,\frac{d}{dy}%
g(V(y)-\beta )=g^{\prime }(V(y)-\beta )\nu _{n}<0,$$ and%
\begin{equation*}
\frac{d^{2}}{dy^{2}}g(V(y)-\beta )=g^{\prime \prime }(V(y)-\beta )\nu
_{n}^{2}>0\ (n\in \mathbf{N});
\end{equation*}%
for $y\in (n,n+\frac{1}{2}),$
$$g(V(y)-\beta )>0,\frac{d}{dy}g(V(y)-\beta
)=g^{\prime }(V(y)-\beta )\nu _{n+1}<0,$$ and
\begin{equation*}
\frac{d^{2}}{dy^{2}}g(V(y)-\beta )=g^{\prime \prime }(V(y)-\beta )\nu
_{n+1}^{2}>0\ (n\in \mathbf{N}).
\end{equation*}

If $g_{1}(u)>0,g_{1}^{\prime }(u)<0,g_{1}^{\prime \prime
}(u)>0,g_{2}(u)>0,g_{2}^{\prime }(u)\leq 0,g_{2}^{\prime \prime }(u)\geq 0,$
then we find for $u>0$ that 
$$g_{1}(u)g_{2}(u)>0,(g_{1}(u)g_{2}(u))^{\prime
}=g_{1}^{\prime }(u)g_{2}(u)+g_{1}(u)g_{2}^{\prime }(u)<0,$$ and
\begin{equation*}
(g_{1}(u)g_{2}(u))^{\prime \prime }=g_{1}^{\prime \prime
}(u)g_{2}(u)+2g_{1}^{\prime }(u)g_{2}^{\prime }(u)+g_{1}(u)g_{2}^{\prime
\prime }(u)>0.
\end{equation*}

(iv) For $\rho >\max \{0,-\alpha \},0<\gamma <\sigma \leq 1,$ we have $$%
h(t)>0,h^{\prime }(t)<0,\ \ h^{\prime \prime }(t)>0,\ \ \text{with}\ k(\sigma )\in
\mathbf{R}_{+},$$ 
and then for $c>0,\beta \leq \frac{\nu _{1}}{2},y\geq \frac{%
1}{2},n\in \mathbf{N}$, we have 
$$h(c(V(y)-\beta ))(V(y)-\beta )^{\sigma
-1}>0,\ \ \frac{d}{dy}[h(c(V(y)-\beta ))(V(y)-\beta )^{\sigma -1}]<0,$$ and%
\begin{equation*}
\frac{d^{2}}{dy^{2}}[h(c(V(y)-\beta ))(V(y)-\beta )^{\sigma -1}]>0\ \ (y\in (n-%
\frac{1}{2},n)\cup (n,n+\frac{1}{2})).
\end{equation*}

Setting $f(y)=h(c(V(y)-\beta ))(V(y)-\beta )^{\sigma -1}$, it follows that $%
f^{\prime }(y)(<0)$ is strictly increasing in $(n-\frac{1}{2},n)$ and
\begin{eqnarray*}
\lim_{x\rightarrow n-}f^{\prime }(y) &=&f^{\prime }(n-0)=[ch^{\prime
}(c(V_{n}-\beta ))(V_{n}-\beta )^{\sigma -1} \\
&&+(\sigma -1)h(c(V_{n}-\beta ))(V_{n}-\beta )^{\sigma -2}]\nu _{n}.
\end{eqnarray*}%
In the same way, for $x\in (n,n+\frac{1}{2}),$ we find that $f^{\prime
}(y)(<0)$ is strictly increasing and%
\begin{eqnarray*}
\lim_{x\rightarrow n+}f^{\prime }(y) &=&f^{\prime }(n+0)=[ch^{\prime
}(c(V_{n}-\beta ))(V_{n}-\beta )^{\sigma -1} \\
&&+(\sigma -1)h(c(V_{n}-\beta ))(V_{n}-\beta )^{\sigma -2}]\nu _{n+1}.
\end{eqnarray*}

In view of $\nu _{n+1}\leq \nu _{n},$ it follows that%
\begin{equation*}
\lim_{x\rightarrow n+}f^{\prime }(x)=f^{\prime }(n+0)\geq f^{\prime
}(n-0)=\lim_{x\rightarrow n-}f^{\prime }(x).
\end{equation*}%
Then by (\ref{2.1}), for $n\in \mathbf{N},$ we have%
\begin{equation}
f(n)<\int_{n-\frac{1}{2}}^{n+\frac{1}{2}}f(y)dy=\int_{n-\frac{1}{2}}^{n+%
\frac{1}{2}}h(c(V(y)-\beta ))(V(y)-\beta )^{\sigma -1}dy.  \label{2.3}
\end{equation}

\begin{lemma}
  If $g(t)(>0)$ is a strictly decreasing continuous function
in $\mathbf{(}\frac{1}{2},\infty ),$ which is strictly convex satisfying 
$$%
\int_{\frac{1}{2}}^{\infty }g(t)dt\in \mathbf{R}_{+},$$ then we have%
\begin{equation}
\int_{1}^{\infty }g(t)dt<\sum_{n=1}^{\infty }g(n)<\int_{\frac{1}{2}}^{\infty
}g(t)dt.  \label{2.4}
\end{equation}
\end{lemma}

\begin{proof}
 By (\ref{2.1}) and the decreasing property, we have
\begin{equation*}
\int_{n}^{n+1}g(t)dt<\int_{n}^{n+1}g(n)dt=g(n)<\int_{n-\frac{1}{2}}^{n+\frac{%
1}{2}}g(t)dt\ \ (n\in \mathbf{N}),
\end{equation*}%
and for $n_{0}\in \mathbf{N},$ it follows that%
\begin{eqnarray*}
\int_{1}^{n_{0}+1}g(t)dt &<&\sum_{n=1}^{n_{0}}g(n)<\sum_{n=1}^{n_{0}}\int_{n-%
\frac{1}{2}}^{n+\frac{1}{2}}g(t)dt=\int_{\frac{1}{2}}^{n_{0}+\frac{1}{2}%
}g(t)dt, \\
\int_{n_{0}+1}^{\infty }g(t)dt &\leq &\sum_{n=n_{0}+1}^{\infty }g(n)\leq
\int_{n_{0}+\frac{1}{2}}^{\infty }g(t)dt<\infty .
\end{eqnarray*}
Hence, we obtain (\ref{2.4}).
\end{proof}

\begin{lemma}
  If $\rho >\max \{0,-\alpha \},0<\gamma <\sigma
\leq 1,$ define the following weight coefficients:%
\begin{eqnarray}
\omega _{\delta }(\sigma ,x) &:&=\sum_{n=1}^{\infty }\frac{\csc h(\rho
U^{\delta \gamma }(x)(V_{n}-\beta )^{\gamma })}{e^{\alpha U^{\delta \gamma
}(x)(V_{n}-\beta )^{\gamma }}}\frac{U^{\delta \sigma }(x)\nu _{n+1}}{%
(V_{n}-\beta )^{1-\sigma }},\ \ x\in \mathbf{R}_{+},  \label{2.5} \\
\varpi _{\delta }(\sigma ,n) &:&=\int_{0}^{\infty }\frac{\csc h(\rho
U^{\delta \gamma }(x)(V_{n}-\beta )^{\gamma })}{e^{\alpha U^{\delta \gamma
}(x)(V_{n}-\beta )^{\gamma }}}\frac{(V_{n}-\beta )^{\sigma }\mu (x)}{%
U^{1-\delta \sigma }(x)}dx,\ \ n\in \mathbf{N}.  \label{2.6}
\end{eqnarray}%
Then, we have the following inequalities:
\begin{eqnarray}
\omega _{\delta }(\sigma ,x) &<&k(\sigma )\ \ (x\in \mathbf{R}_{+}),  \label{2.7}
\\
\varpi _{\delta }(\sigma ,n) &\leq &k(\sigma )\ \ (n\in \mathbf{N}),  \label{2.8}
\end{eqnarray}%
where, $k(\sigma )$ is indicated by (\ref{2.2}).
\end{lemma}

\begin{proof}
  Since $V_{n}=V(n),$ and for $t\in (n-\frac{1}{2},n),$ $$\nu
_{n+1}\leq \nu _{n}=V^{\prime }(t);$$ for $t\in (n,n+\frac{1}{2}),$ $$\nu
_{n+1}=V^{\prime }(t),$$ by (\ref{2.3}) (for $c=U^{\delta }(x)$), we have
\begin{eqnarray*}
&&\frac{\csc h(\rho U^{\delta \gamma }(x)(V_{n}-\beta )^{\gamma })}{%
e^{\alpha U^{\delta \gamma }(x)(V_{n}-\beta )^{\gamma }}}\frac{U^{\delta
\sigma }(x)}{(V_{n}-\beta )^{1-\sigma }} \\
&=&\frac{\csc h(\rho U^{\delta \gamma }(x)(V(n)-\beta )^{\gamma })}{%
e^{\alpha U^{\delta \gamma }(x)(V(n)-\beta )^{\gamma }}}\frac{U^{\delta
\sigma }(x)}{(V(n)-\beta )^{1-\sigma }} \\
&<&\int_{n-\frac{1}{2}}^{n+\frac{1}{2}}\frac{\csc h(\rho U^{\delta \gamma
}(x)(V(t)-\beta )^{\gamma })}{e^{\alpha U^{\delta \gamma }(x)(V(t)-\beta
)^{\gamma }}}\frac{U^{\delta \sigma }(x)}{(V(t)-\beta )^{1-\sigma }}dt\ \ (n\in
\mathbf{N}),
\end{eqnarray*}%
\begin{eqnarray*}
\omega _{\delta }(\sigma ,x) &<&\sum_{n=1}^{\infty }\nu _{n+1}\int_{n-\frac{1%
}{2}}^{n+\frac{1}{2}}\frac{\csc h(\rho U^{\delta \gamma }(x)(V(t)-\beta
)^{\gamma })}{e^{\alpha U^{\delta \gamma }(x)(V(t)-\beta )^{\gamma }}}\frac{%
U^{\delta \sigma }(x)dt}{(V(t)-\beta )^{1-\sigma }} \\
&\leq &\sum_{n=1}^{\infty }\int_{n-\frac{1}{2}}^{n+\frac{1}{2}}\frac{\csc
h(\rho U^{\delta \gamma }(x)(V(t)-\beta )^{\gamma })}{e^{\alpha U^{\delta
\gamma }(x)(V(t)-\beta )^{\gamma }}}\frac{U^{\delta \sigma }(x)V^{\prime }(t)%
}{(V(t)-\beta )^{1-\sigma }}dt \\
&=&\int_{\frac{1}{2}}^{\infty }\frac{\csc h(\rho U^{\delta \gamma
}(x)(V(t)-\beta )^{\gamma })}{e^{\alpha U^{\delta \gamma }(x)(V(t)-\beta
)^{\gamma }}}\frac{U^{\delta \sigma }(x)V^{\prime }(t)}{(V(t)-\beta
)^{1-\sigma }}dt.
\end{eqnarray*}%
Setting $u=U^{\delta }(x)(V(t)-\beta ),$ by (\ref{2.2}), we obtain%
\begin{eqnarray*}
\omega _{\delta }(\sigma ,x) &<&\int_{U^{\delta }(x)(\frac{\nu _{1}}{2}%
-\beta )}^{U^{\delta }(x)V(\infty )}\frac{\csc h(\rho u^{\gamma })}{%
e^{\alpha u^{\gamma }}}\frac{U^{\delta \sigma }(x)U^{-\delta }(x)}{%
(uU^{-\delta }(x))^{1-\sigma }}du \\
&\leq &\int_{0}^{\infty }\frac{\csc h(\rho u^{\gamma })}{e^{\alpha u^{\gamma
}}}u^{\sigma -1}du=k(\sigma ).
\end{eqnarray*}%
Hence, (\ref{2.7}) follows.

Setting $u=(V_{n}-\beta )U^{\delta }(x)$ in (\ref{2.6}), we find $du=\delta
(V_{n}-\beta )U^{\delta -1}(x)\mu (x)dx$ and
\begin{eqnarray*}
\varpi _{\delta }(\sigma ,n) &=&\frac{1}{\delta }\int_{(V_{n}-\beta
)U^{\delta }(0)}^{(V_{n}-\beta )U^{\delta }(\infty )}\frac{\csc h(\rho
u^{\gamma })}{e^{\alpha u^{\gamma }}}\frac{(V_{n}-\beta )^{\sigma
-1}[(V_{n}-\beta )^{-1}u]^{\frac{1}{\delta }-1}}{[(V_{n}-\beta )^{-1}u]^{%
\frac{1}{\delta }-\sigma }}du \\
&=&\frac{1}{\delta }\int_{(V_{n}-\beta )U^{\delta }(0)}^{(V_{n}-\beta
)U^{\delta }(\infty )}\frac{\csc h(\rho u^{\gamma })}{e^{\alpha u^{\gamma }}}%
u^{\sigma -1}du.
\end{eqnarray*}

\noindent If $\delta =1,$ then%
\begin{equation*}
\varpi _{1}(\sigma ,n)=\int_{0}^{(V_{n}-\beta )U(\infty )}\frac{\csc h(\rho
u^{\gamma })}{e^{\alpha u^{\gamma }}}u^{\sigma -1}du\leq \int_{0}^{\infty }%
\frac{\csc h(\rho u^{\gamma })}{e^{\alpha u^{\gamma }}}u^{\sigma -1}du;
\end{equation*}%
if $\delta =-1,$ then%
\begin{equation*}
\varpi _{-1}(\sigma ,n)=-\int_{\infty }^{(V_{n}-\beta )U^{-1}(\infty )}\frac{%
\csc h(\rho u^{\gamma })}{e^{\alpha u^{\gamma }}}u^{\sigma -1}du\leq
\int_{0}^{\infty }\frac{\csc h(\rho u^{\gamma })}{e^{\alpha u^{\gamma }}}%
u^{\sigma -1}du.
\end{equation*}%
Hence, by (\ref{2.2}), we have (\ref{2.8}).
\end{proof}

\begin{remark}
  We do not need the condition of $\sigma \leq 1$ in
obtaining (\ref{2.8}). If $U(\infty )=\infty ,$ then we have
\begin{equation}
\varpi _{\delta }(\sigma ,n)=k(\sigma )\ \ (n\in \mathbf{N}).  \label{2.9}
\end{equation}
\end{remark}

For example, we set $\mu (t)=\frac{1}{(1+t)^{a}}\ (t>0;0\leq a\leq 1),$ then
for $x\geq 0,$ we find
\begin{equation*}
U(x)=\int_{0}^{x}\frac{dt}{(1+t)^{a}}=\left\{
\begin{array}{c}
\frac{(1+x)^{1-a}-1}{1-a},0\leq a<1 \\
\ln (1+x),a=1%
\end{array}%
\right. <\infty ,
\end{equation*}%
$U(0)=0$ and $U(\infty )=\int_{0}^{\infty }\frac{dt}{(1+t)^{a}}=\infty .$

\begin{lemma}
  If $\rho >\max \{0,-\alpha \},0<\gamma <\sigma \leq
1,V(\infty )=\infty ,$ then, (i) for $x\in \mathbf{R}_{+}\mathbf{,}$ we have
\begin{equation}
k(\sigma )(1-\theta _{\delta }(\sigma ,x))<\omega _{\delta }(\sigma ,x),
\label{2.10}
\end{equation}%
where,
\begin{eqnarray*}
\theta _{\delta }(\sigma ,x) &:=&\frac{1}{k(\sigma )}\int_{0}^{U^{\delta
}(x)(\nu _{1}-\beta )}\frac{\csc h(\rho u^{\gamma })}{e^{\alpha u^{\gamma }}}%
u^{\sigma -1}du \\
&=&O((U(x))^{\frac{\delta }{2}(\sigma -\gamma )})\in (0,1);
\end{eqnarray*}%
(ii) for any $b>0,$ we have%
\begin{equation}
\sum_{n=1}^{\infty }\frac{\nu _{n+1}}{(V_{n}-\beta )^{1+b}}=\frac{1}{b}\left[
\frac{1}{(\nu _{1}-\beta )^{b}}+bO(1)\right] .  \label{2.11}
\end{equation}
\end{lemma}

\begin{proof}
  By (\ref{2.4}), we find%
\begin{eqnarray*}
\omega _{\delta }(\sigma ,x) &>&\sum_{n=1}^{\infty }\nu _{n+1}\int_{n}^{n+1}%
\frac{\csc h(\rho U^{\delta \gamma }(x)(V(t)-\beta )^{\gamma })}{e^{\alpha
U^{\delta \gamma }(x)(V(t)-\beta )^{\gamma }}}\frac{U^{\delta \sigma }(x)dt}{%
(V(t)-\beta )^{1-\sigma }} \\
&=&\sum_{n=1}^{\infty }\int_{n}^{n+1}\frac{\csc h(\rho U^{\delta \gamma
}(x)(V(t)-\beta )^{\gamma })}{e^{\alpha U^{\delta \gamma }(x)(V(t)-\beta
)^{\gamma }}}\frac{U^{\delta \sigma }(x)V^{\prime }(t)}{(V(t)-\beta
)^{1-\sigma }}dt \\
&=&\int_{1}^{\infty }\frac{\csc h(\rho U^{\delta \gamma }(x)(V(t)-\beta
)^{\gamma })}{e^{\alpha U^{\delta \gamma }(x)(V(t)-\beta )^{\gamma }}}\frac{%
U^{\delta \sigma }(x)V^{\prime }(t)}{(V(t)-\beta )^{1-\sigma }}dt.
\end{eqnarray*}%
Setting $u=U^{\delta }(x)(V(t)-\beta ),$ in view of $V(\infty )=\infty ,$ by
(\ref{2.2}), we find%
\begin{eqnarray*}
\omega _{\delta }(\sigma ,x) &>&\int_{U^{\delta }(x)(V(1)-\beta )}^{\infty }%
\frac{\csc h(\rho u^{\gamma })}{e^{\alpha u^{\gamma }}}u^{\sigma -1}du \\
&=&k(\sigma )-\int_{0}^{U^{\delta }(x)(\nu _{1}-\beta )}\frac{\csc h(\rho
u^{\gamma })}{e^{\alpha u^{\gamma }}}u^{\sigma -1}du \\
&=&k(\sigma )(1-\theta _{\delta }(\sigma ,x)).
\end{eqnarray*}

\noindent Since 
$$F(u)=\frac{\csc h(\rho u^{\gamma })}{e^{\alpha u^{\gamma }}}$$ 
is
continuous in $(0,\infty )$ satisfying $u^{\frac{1}{2}(\sigma +\gamma
)}F(u)\rightarrow 0\ (u\rightarrow 0^{+}),$ and $u^{\frac{1}{2}(\sigma +\gamma
)}F(u)\rightarrow 0\ (u\rightarrow \infty ),$ there exists a constant $L>0,$
such that $u^{\frac{1}{2}(\sigma +\gamma )}F(u)\leq L,$ namely,%
\begin{equation*}
\frac{\csc h(\rho u^{\gamma })}{e^{\alpha u^{\gamma }}}\leq Lu^{\frac{-1}{2}%
(\sigma +\gamma )}\ \ (u\in (0,\infty )).
\end{equation*}
Hence we find%
\begin{eqnarray*}
0 &<&\theta _{\delta }(\sigma ,x)\leq \frac{L}{k(\sigma )}%
\int_{0}^{U^{\delta }(x)(\nu _{1}-\beta )}u^{\frac{1}{2}(\sigma -\gamma
)-1}du \\
&=&\frac{2L[U^{\delta }(x)(\nu _{1}-\beta )]^{\frac{1}{2}(\sigma -\gamma )}}{%
k(\sigma )(\sigma -\gamma )}\ \ (x\in \mathbf{R}_{+}),
\end{eqnarray*}%
and then (\ref{2.10}) follows.

For $b>0,$ we find
\begin{eqnarray*}
\sum_{n=1}^{\infty }\frac{\nu _{n+1}}{(V_{n}-\beta )^{1+b}}
&<&\frac{\nu _{2}}{(V_{1}-\beta )^{1+b}}+\sum_{n=2}^{\infty }\int_{n-\frac{1%
}{2}}^{n+\frac{1}{2}}\frac{V^{\prime }(x)}{(V(x)-\beta )^{1+b}}dx \\
&=&\frac{\nu _{2}}{(\nu _{1}-\beta )^{1+b}}+\int_{\frac{3}{2}}^{\infty }%
\frac{V^{\prime }(x)}{(V(x)-\beta )^{1+b}}dx \\
&=&\frac{\nu _{2}}{(\nu _{1}-\beta )^{1+b}}+\int_{\nu _{1}+\frac{1}{2}\nu
_{2}-\beta }^{\infty }\frac{du}{u^{1+b}} \\
&\leq &\frac{1}{b}\left[ \frac{1}{(\nu _{1}-\beta )^{b}}+b\frac{\nu _{2}}{%
(\nu _{1}-\beta )^{1+b}}\right] ,
\end{eqnarray*}%
\begin{eqnarray*}
\sum_{n=1}^{\infty }\frac{\nu _{n+1}}{(V_{n}-\beta )^{1+b}}
&=&\sum_{n=1}^{\infty }\int_{n}^{n+1}\frac{\nu _{n+1}}{(V(n)-\beta )^{1+b}}%
dx>\sum_{n=1}^{\infty }\int_{n}^{n+1}\frac{V^{\prime }(x)dx}{(V(x)-\beta
)^{1+b}} \\
&=&\int_{1}^{\infty }\frac{V^{\prime }(x)dx}{(V(x)-\beta )^{1+b}}=\frac{1}{%
b(\nu _{1}-\beta )^{b}}.
\end{eqnarray*}%
Hence we have (\ref{2.11}).
\end{proof}

\noindent \textbf{Note}. For example, $\nu _{n}=\frac{1}{n^{a}}\ (n\in \mathbf{N};0\leq
a\leq 1)$ satisfies the condition  that \mbox{$\nu _{n}>0\ (n\in \mathbf{N}%
),$} $\{\nu _{n}\}_{n=1}^{\infty }$ is decreasing, and $V(\infty )=\infty $.

\section{Main Results and Operator Expressions}

\begin{theorem}
  If $\rho >\max \{0,-\alpha \},0<\gamma <\sigma
\leq 1,$ $k(\sigma )$ is indicated by (\ref{2.2}), then for $p>1,$ $%
0<||f||_{p,\Phi _{\delta }},||a||_{q,\Psi _{\beta }}<\infty ,$ we have the
following equivalent inequalities:%
\begin{eqnarray}
I &:=&\sum_{n=1}^{\infty }\int_{0}^{\infty }\frac{\csc h(\rho U^{\delta
\gamma }(x)(V_{n}-\beta )^{\gamma })}{e^{\alpha U^{\delta \gamma
}(x)(V_{n}-\beta )^{\gamma }}}a_{n}f(x)dx<k(\sigma )||f||_{p,\Phi _{\delta
}}||a||_{q,\Psi _{\beta }},  \label{3.1} \\
J_{1} &:=&\sum_{n=1}^{\infty }\frac{\nu _{n+1}}{(V_{n}-\beta )^{1-p\sigma }}%
\left[ \int_{0}^{\infty }\frac{\csc h(\rho U^{\delta \gamma }(x)(V_{n}-\beta
)^{\gamma })}{e^{\alpha U^{\delta \gamma }(x)(V_{n}-\beta )^{\gamma }}}f(x)dx%
\right] ^{p}  \notag \\
&<&k(\sigma )||f||_{p,\Phi _{\delta }},  \label{3.2}
\end{eqnarray}%
\begin{eqnarray}
J_{2} &:=&\left\{ \int_{0}^{\infty }\frac{\mu (x)}{U^{1-q\delta \sigma }(x)}%
\left[ \sum_{n=1}^{\infty }\frac{\csc h(\rho U^{\delta \gamma
}(x)(V_{n}-\beta )^{\gamma })}{e^{\alpha U^{\delta \gamma }(x)(V_{n}-\beta
)^{\gamma }}}a_{n}\right] ^{q}dx\right\} ^{\frac{1}{q}}  \notag \\
&<&k(\sigma )||a||_{q,\Psi _{\beta }}.  \label{3.3}
\end{eqnarray}
\end{theorem}

\begin{proof}
 By the weighted H\"{o}lder inequality (cf. \cite{K1}), we have%
\begin{eqnarray}
&&\left[ \int_{0}^{\infty }\frac{\csc h(\rho U^{\delta \gamma
}(x)(V_{n}-\beta )^{\gamma })}{e^{\alpha U^{\delta \gamma }(x)(V_{n}-\beta
)^{\gamma }}}f(x)dx\right] ^{p} \notag\\
&=&\left[ \int_{0}^{\infty }\frac{\csc h(\rho U^{\delta \gamma
}(x)(V_{n}-\beta )^{\gamma })}{e^{\alpha U^{\delta \gamma }(x)(V_{n}-\beta
)^{\gamma }}}\right.\notag\\
&&\left. \times \frac{U^{\frac{1-\delta \sigma }{q}}(x)f(x)}{(V_{n}-\beta )^{%
\frac{1-\sigma }{p}}\mu ^{\frac{1}{q}}(x)}\cdot \frac{(V_{n}-\beta )^{\frac{%
1-\sigma }{p}}\mu ^{\frac{1}{q}}(x)}{U^{\frac{1-\delta \sigma }{q}}(x)}dx%
\right] ^{p}  \notag \\
&\leq &\int_{0}^{\infty }\frac{\csc h(\rho U^{\delta \gamma }(x)(V_{n}-\beta
)^{\gamma })}{e^{\alpha U^{\delta \gamma }(x)(V_{n}-\beta )^{\gamma }}}%
\left[ \frac{U^{\frac{p(1-\delta \sigma )}{q}}(x)f^{p}(x)}{(V_{n}-\beta
)^{1-\sigma }\mu ^{\frac{p}{q}}(x)}\right] dx  \notag \\
&&\times \left[ \int_{0}^{\infty }\frac{\csc h(\rho U^{\delta \gamma
}(x)(V_{n}-\beta )^{\gamma })}{e^{\alpha U^{\delta \gamma }(x)(V_{n}-\beta
)^{\gamma }}}\frac{(V_{n}-\beta )^{(1-\sigma )(p-1)}\mu (x)}{U^{1-\delta
\sigma }(x)}dx\right] ^{p-1}  \notag \\
&=&\frac{(\varpi _{\delta }(\sigma ,n))^{p-1}}{(V_{n}-\beta )^{p\sigma
-1}\nu _{n+1}} \notag \\
&&\times\int_{0}^{\infty }\frac{\csc h(\rho U^{\delta \gamma
}(x)(V_{n}-\beta )^{\gamma })}{e^{\alpha U^{\delta \gamma }(x)(V_{n}-\beta
)^{\gamma }}}\frac{U^{(1-\delta \sigma )(p-1)}(x)\nu _{n+1}f^{p}(x)}{%
(V_{n}-\beta )^{1-\sigma }\mu ^{p-1}(x)}dx.  \label{3.4}
\end{eqnarray}

In view of (\ref{2.8}) and the Lebesgue term by term integration theorem (cf.
\cite{K2}), we find%
\begin{eqnarray}
J_{1} &\leq &(k(\sigma ))^{\frac{1}{q}}\left[ \sum_{n=1}^{\infty
}\int_{0}^{\infty }\frac{\csc h(\rho U^{\delta \gamma }(x)(V_{n}-\beta
)^{\gamma })}{e^{\alpha U^{\delta \gamma }(x)(V_{n}-\beta )^{\gamma }}}\frac{%
U^{(1-\delta \sigma )(p-1)}(x)\nu _{n+1}}{(V_{n}-\beta )^{1-\sigma }\mu
^{p-1}(x)}f^{p}(x)dx\right] ^{\frac{1}{p}}  \notag \\
&=&(k(\sigma ))^{\frac{1}{q}}\left[ \int_{0}^{\infty }\sum_{n=1}^{\infty }%
\frac{\csc h(\rho U^{\delta \gamma }(x)(V_{n}-\beta )^{\gamma })}{e^{\alpha
U^{\delta \gamma }(x)(V_{n}-\beta )^{\gamma }}}\frac{U^{(1-\delta \sigma
)(p-1)}(x)\nu _{n+1}}{(V_{n}-\beta )^{1-\sigma }\mu ^{p-1}(x)}f^{p}(x)dx%
\right] ^{\frac{1}{p}}  \notag \\
&=&(k(\sigma ))^{\frac{1}{q}}\left[ \int_{0}^{\infty }\omega _{\delta
}(\sigma ,x)\frac{U^{p(1-\delta \sigma )-1}(x)}{\mu ^{p-1}(x)}f^{p}(x)dx%
\right] ^{\frac{1}{p}}.  \label{3.5}
\end{eqnarray}%
Then by (\ref{2.7}), we derive (\ref{3.2}).

By H\"{o}lder's inequality (cf. \cite{K1}), we have%
\begin{eqnarray}
I &=&\sum_{n=1}^{\infty }\left[ \frac{\nu _{n+1}^{\frac{1}{p}}}{(V_{n}-\beta
)^{\frac{1}{p}-\sigma }}\int_{0}^{\infty }\frac{\csc h(\rho U^{\delta \gamma
}(x)(V_{n}-\beta )^{\gamma })}{e^{\alpha U^{\delta \gamma }(x)(V_{n}-\beta
)^{\gamma }}}f(x)dx\right]   \notag \\
&=& \left[ \frac{(V_{n}-\beta )^{\frac{1}{p}-\sigma }a_{n}}{\nu
_{n+1}^{\frac{1}{p}}}\right] \leq J_{1}||a||_{q,\Psi _{\beta }}.  \label{3.6}
\end{eqnarray}%
Then by (\ref{3.2}), we obtain (\ref{3.1}). On the other hand, assuming that (%
\ref{3.1}) is valid, we set%
\begin{equation*}
a_{n}:=\frac{\nu _{n+1}}{(V_{n}-\beta )^{1-p\sigma }}\left[ \int_{0}^{\infty
}\frac{\csc h(\rho U^{\delta \gamma }(x)(V_{n}-\beta )^{\gamma })}{e^{\alpha
U^{\delta \gamma }(x)(V_{n}-\beta )^{\gamma }}}f(x)dx\right] ^{p-1},\ \ n\in
\mathbf{N}.
\end{equation*}%
Then we find $J_{1}^{p}=||a||_{q,\Psi _{\beta }}^{q}.$ If $J_{1}=0,$ then (%
\ref{3.2}) is trivially valid; if $J_{1}=\infty ,$ then (\ref{3.2}) is still not valid. Suppose that $0<J_{1}<\infty .$ By (\ref{3.1}), we have%
\begin{eqnarray*}
||a||_{q,\Psi _{\beta }}^{q} &=&J_{1}^{p}=I<k(\sigma )||f||_{p,\Phi _{\delta
}}||a||_{q,\Psi _{\beta }}, \\
||a||_{q,\Psi _{\beta }}^{q-1} &=&J_{1}<k(\sigma )||f||_{p,\Phi _{\delta }},
\end{eqnarray*}%
and then (\ref{3.2}) follows, which is equivalent to (\ref{3.1}).

Still by the weighted H\"{o}lder inequality (cf. \cite{K1}), we have%
\begin{eqnarray}
&&\left[ \sum_{n=1}^{\infty }\frac{\csc h(\rho U^{\delta \gamma
}(x)(V_{n}-\beta )^{\gamma })}{e^{\alpha U^{\delta \gamma }(x)(V_{n}-\beta
)^{\gamma }}}a_{n}\right] ^{q} \notag\\
&=&\left[ \sum_{n=1}^{\infty }\frac{\csc h(\rho U^{\delta \gamma
}(x)(V_{n}-\beta )^{\gamma })}{e^{\alpha U^{\delta \gamma }(x)(V_{n}-\beta
)^{\gamma }}}\cdot \frac{U^{\frac{1-\delta \sigma }{q}}(x)\nu _{n+1}^{\frac{1%
}{p}}}{(V_{n}-\beta )^{\frac{1-\sigma }{p}}}\cdot \frac{(V_{n}-\beta )^{%
\frac{1-\sigma }{p}}a_{n}}{U^{\frac{1-\delta \sigma }{q}}(x)\nu _{n+1}^{%
\frac{1}{p}}}\right] ^{q}\notag\\
&\leq &\left[ \sum_{n=1}^{\infty }\frac{\csc h(\rho U^{\delta \gamma
}(x)(V_{n}-\beta )^{\gamma })}{e^{\alpha U^{\delta \gamma }(x)(V_{n}-\beta
)^{\gamma }}}\frac{U^{(1-\delta \sigma )(p-1)}(x)\nu _{n+1}}{(V_{n}-\beta
)^{1-\sigma }}\right] ^{q-1}  \notag \\
&&\times \sum_{n=1}^{\infty }\frac{\csc h(\rho U^{\delta \gamma
}(x)(V_{n}-\beta )^{\gamma })}{e^{\alpha U^{\delta \gamma }(x)(V_{n}-\beta
)^{\gamma }}}\frac{(V_{n}-\beta )^{\frac{q(1-\sigma )}{p}}}{U^{1-\delta
\sigma }(x)\nu _{n+1}^{q-1}}a_{n}^{q}  \notag \\
&=&\frac{(\omega _{\delta }(\sigma ,x))^{q-1}}{U^{q\delta \sigma -1}(x)\mu
(x)}\sum_{n=1}^{\infty }\frac{\csc h(\rho U^{\delta \gamma }(x)(V_{n}-\beta
)^{\gamma })}{e^{\alpha U^{\delta \gamma }(x)(V_{n}-\beta )^{\gamma }}}\frac{%
(V_{n}-\beta )^{(1-\sigma )(q-1)}\mu (x)}{U^{1-\delta \sigma }(x)\nu
_{n+1}^{q-1}}a_{n}^{q}.  \label{3.7}
\end{eqnarray}%
Then by (\ref{2.7}) and the Lebesgue term by term integration theorem (cf. \cite%
{K2}), it follows that%
\begin{equation*}
J_{2}<(k(\sigma ))^{\frac{1}{p}}\left[ \int_{0}^{\infty }\sum_{n=1}^{\infty }%
\frac{\csc h(\rho U^{\delta \gamma }(x)(V_{n}-\beta )^{\gamma })}{e^{\alpha
U^{\delta \gamma }(x)(V_{n}-\beta )^{\gamma }}}\frac{(V_{n}-\beta
)^{(1-\sigma )(q-1)}\mu (x)}{U^{1-\delta \sigma }(x)\nu _{n+1}^{q-1}}%
a_{n}^{q}dx\right] ^{\frac{1}{q}}
\end{equation*}%
\begin{eqnarray}
&=&(k(\sigma ))^{\frac{1}{p}}\left[ \sum_{n=1}^{\infty }\int_{0}^{\infty }%
\frac{\csc h(\rho U^{\delta \gamma }(x)(V_{n}-\beta )^{\gamma })}{e^{\alpha
U^{\delta \gamma }(x)(V_{n}-\beta )^{\gamma }}}\frac{(V_{n}-\beta
)^{(1-\sigma )(q-1)}\mu (x)}{U^{1-\delta \sigma }(x)\nu _{n+1}^{q-1}}%
a_{n}^{q}dx\right] ^{\frac{1}{q}}  \notag \\
&=&(k(\sigma ))^{\frac{1}{p}}\left[ \sum_{n=1}^{\infty }\varpi _{\delta
}(\sigma ,n)\frac{(V_{n}-\beta )^{q(1-\sigma )-1}}{\nu _{n+1}^{q-1}}a_{n}^{q}%
\right] ^{\frac{1}{q}}.  \label{3.8}
\end{eqnarray}%
Then by (\ref{2.8}), \ we derive (\ref{3.3}).

By H\"{o}lder's inequality (cf. \cite{K1}), we have%
\begin{eqnarray}
I &=&\int_{0}^{\infty }\left( \frac{U^{\frac{1}{q}-\delta \sigma }(x)}{\mu ^{%
\frac{1}{q}}(x)}f(x)\right) \left[ \frac{\mu ^{\frac{1}{q}}(x)}{U^{\frac{1}{q%
}-\delta \sigma }(x)}\sum_{n=1}^{\infty }\frac{\csc h(\rho U^{\delta \gamma
}(x)(V_{n}-\beta )^{\gamma })}{e^{\alpha U^{\delta \gamma }(x)(V_{n}-\beta
)^{\gamma }}}a_{n}\right] dx  \notag \\
&\leq &||f||_{p,\Phi _{\delta }}J_{2}.  \label{3.9}
\end{eqnarray}%
Then by (\ref{3.3}), we obtain (\ref{3.1}). On the other hand, assuming that (%
\ref{3.3}) is valid, we set%
\begin{equation*}
f(x):=\frac{\mu (x)}{U^{1-q\delta \sigma }(x)}\left[ \sum_{n=1}^{\infty }%
\frac{\csc h(\rho U^{\delta \gamma }(x)(V_{n}-\beta )^{\gamma })}{e^{\alpha
U^{\delta \gamma }(x)(V_{n}-\beta )^{\gamma }}}a_{n}\right] ^{q-1},\ \ x\in
\mathbf{R}_{+}.
\end{equation*}%
Then we find $J_{2}^{q}=||f||_{p,\Phi _{\delta }}^{p}.$ If $J_{2}=0,$ then (%
\ref{3.3}) is trivially valid; if $J_{2}=\infty ,$ then (\ref{3.3}) remains
impossible. Suppose that $0<J_{2}<\infty .$ By (\ref{3.1}), we have
\begin{eqnarray*}
||f||_{p,\Phi _{\delta }}^{p} &=&J_{2}^{q}=I<k(\sigma )||f||_{p,\Phi
_{\delta }}||a||_{q,\Psi _{\beta }}, \\
||f||_{p,\Phi _{\delta }}^{p-1} &=&J_{2}<k(\sigma )||a||_{q,\Psi _{\beta }},
\end{eqnarray*}%
and then (\ref{3.3}) follows, which is equivalent to (\ref{3.1}).

Therefore, (\ref{3.1}), (\ref{3.2}) and (\ref{3.3}) are equivalent.
\end{proof}

\begin{theorem}
  With the assumptions of Theorem 3.1, if $U(\infty
)=V(\infty )=\infty ,$ then the constant factor $k(\sigma )$ in (\ref{3.1}),
(\ref{3.2}) and (\ref{3.3}) is the best possible.
\end{theorem}

\begin{proof}
 For $\varepsilon \in (0,\frac{q(\sigma -\gamma )}{2})),$ we
set $\widetilde{\sigma }=\sigma -\frac{\varepsilon }{q},$ and $\widetilde{f}=%
\widetilde{f}(x),\ x\in \mathbf{R}_{+},\ \widetilde{a}=\{\widetilde{a}%
_{n}\}_{n=1}^{\infty },$%
\begin{eqnarray}
\widetilde{f}(x) &=&\left\{
\begin{array}{c}
U^{\delta (\widetilde{\sigma }+\varepsilon )-1}(x)\mu (x),\ 0<x^{\delta }\leq 1
\\
0,\ x^{\delta }>0%
\end{array}%
\right. ,  \label{3.10} \\
\widetilde{a}_{n} &=&(V_{n}-\beta )^{\widetilde{\sigma }-1}\nu
_{n+1}=(V_{n}-\beta )^{\sigma -\frac{\varepsilon }{q}-1}\nu _{n+1},\ n\in
\mathbf{N}.  \label{3.11}
\end{eqnarray}%
Then for $\delta =\pm 1,$ since $U(\infty )=\infty ,$ we find%
\begin{equation}
\int_{\{x>0;0<x^{\delta }\leq 1\}}\frac{\mu (x)}{U^{1-\delta \varepsilon }(x)%
}dx=\frac{1}{\varepsilon }U^{\delta \varepsilon }(1).  \label{3.12}
\end{equation}%
By (\ref{2.11}), (\ref{3.12}) and (\ref{2.10}), we obtain%
\begin{eqnarray}
||\widetilde{f}||_{p,\Phi _{\delta }}||\widetilde{a}||_{q,\Psi_{\beta} } &=&\left(
\int_{\{x>0;0<x^{\delta }\leq 1\}}\frac{\mu (x)dx}{U^{1-\delta \varepsilon
}(x)}\right) ^{\frac{1}{p}}\left[ \sum_{n=1}^{\infty }\frac{\nu _{n+1}}{%
(V_{n}-\beta )^{1+\varepsilon }}\right] ^{\frac{1}{q}}  \notag \\
&=&\frac{1}{\varepsilon }U^{\frac{\delta \varepsilon }{p}}(1)\left[ \frac{1}{%
(\nu _{1}-\beta )^{\varepsilon }}+\varepsilon O(1)\right] ^{\frac{1}{q}},
\label{3.13}
\end{eqnarray}%
\begin{eqnarray*}
\widetilde{I} &:&=\int_{0}^{\infty }\sum_{n=1}^{\infty }\frac{\csc h(\rho
U^{\delta \gamma }(x)(V_{n}-\beta )^{\gamma })}{e^{\alpha U^{\delta \gamma
}(x)(V_{n}-\beta )^{\gamma }}}\widetilde{a}_{n}\widetilde{f}(x)dx \\
&=&\int_{\{x>0;0<x^{\delta }\leq 1\}}\sum_{n=1}^{\infty }\frac{\csc h(\rho
U^{\delta \gamma }(x)(V_{n}-\beta )^{\gamma })}{e^{\alpha U^{\delta \gamma
}(x)(V_{n}-\beta )^{\gamma }}}\frac{(V_{n}-\beta )^{\widetilde{\sigma }%
-1}\nu _{n+1}\mu (x)}{U^{1-\delta (\widetilde{\sigma }+\varepsilon )}(x)}dx
\\
&=&\int_{\{x>0;0<x^{\delta }\leq 1\}}\omega _{\delta }(\widetilde{\sigma },x)%
\frac{\mu (x)}{U^{1-\delta \varepsilon }(x)}dx \\
&\geq &k(\widetilde{\sigma })\int_{\{x>0;0<x^{\delta }\leq 1\}}(1-\theta
_{\delta }(\widetilde{\sigma },x))\frac{\mu (x)}{U^{1-\delta \varepsilon }(x)%
}dx \\
&=&k(\widetilde{\sigma })\int_{\{x>0;0<x^{\delta }\leq
1\}}(1-O((U(x))^{\delta \frac{\sigma -\frac{\varepsilon }{q}-\gamma }{2}}))%
\frac{\mu (x)}{U^{1-\delta \varepsilon }(x)}dx\\
&=&k(\widetilde{\sigma })\left[ \int_{\{x>0;0<x^{\delta }\leq 1\}}\frac{\mu
(x)dx}{U^{1-\delta \varepsilon }(x)}-\int_{\{x>0;0<x^{\delta }\leq 1\}}O\left(%
\frac{\mu (x)}{U^{1-\delta (\varepsilon +\frac{\sigma -\frac{\varepsilon }{q}%
-\gamma }{2})}(x)}\right)dx\right]  \\
&=&\frac{1}{\varepsilon }k(\sigma -\frac{\varepsilon }{q})(U^{\delta
\varepsilon }(1)-\varepsilon O_{1}(1)).
\end{eqnarray*}

If there exists a positive constant $K\leq k(\sigma ),$ such that (\ref{3.1}%
) is valid when replacing $k(\sigma )$ to $K,$ then in particular, by
the Lebesgue term by term integration theorem, we have $$\varepsilon \widetilde{I}%
<\varepsilon K||\widetilde{f}||_{p,\Phi _{\delta }}||\widetilde{a}||_{q,\Psi_{\beta}
},$$ namely,%
\begin{equation*}
k(\sigma -\frac{\varepsilon }{q})(U^{\delta \varepsilon }(1)-\varepsilon
O_{1}(1))<K\cdot U^{\frac{\delta \varepsilon }{p}}(1)\left[ \frac{1}{(\nu
_{1}-\beta )^{\varepsilon }}+\varepsilon O(1)\right] ^{\frac{1}{q}}.
\end{equation*}%
It follows that $k(\sigma )\leq K(\varepsilon \rightarrow 0^{+}).$ Hence, $%
K=k(\sigma )$ is the best possible constant factor of (\ref{3.1}).

The constant factor $k(\sigma )$ in (\ref{3.2}) ((\ref{3.3})) is still the
best possible. Otherwise, we would reach a contradiction by (\ref{3.6}) ((%
\ref{3.9})) that the constant factor in (\ref{3.1}) is not the best
possible.
\end{proof}

For $p>1,$ we obtain 
$$\Psi _{\beta }^{1-p}(n)=\frac{\nu _{n+1}}{(V_{n}-\beta
)^{1-p\sigma }}\ \ (n\in \mathbf{N}),\ \ \Phi _{\delta }^{1-q}(x)=\frac{\mu (x)}{%
U^{1-q\delta \sigma }(x)}\ \ (x\in \mathbf{R}_{+}),$$ and define the following
real normed spaces:%
\begin{eqnarray*}
L_{p,\Phi _{\delta }}(\mathbf{R}_{+}) &=&\{f;f=f(x),x\in \mathbf{R}%
_{+},||f||_{p,\Phi _{\delta }}<\infty \}, \\
l_{q,\Psi _{\beta }} &=&\{a;a=\{a_{n}\}_{n=1}^{\infty },||a||_{q,\Psi
_{\beta }}<\infty \}, \\
L_{q,\Phi _{\delta }^{1-q}}(\mathbf{R}_{+}) &=&\{h;h=h(x),x\in \mathbf{R}%
_{+},||h||_{q,\Phi _{\delta }^{1-q}}<\infty \}, \\
l_{p,\Psi _{\beta }^{1-p}} &=&\{c;c=\{c_{n}\}_{n=1}^{\infty },||c||_{p,\Psi
_{\beta }^{1-p}}<\infty \}.
\end{eqnarray*}

\noindent Assuming that $f\in L_{p,\Phi _{\delta }}\ (\mathbf{R}_{+})$ and setting
\begin{equation*}
c=\{c_{n}\}_{n=1}^{\infty },\ \ c_{n}:=\int_{0}^{\infty }\frac{\csc h(\rho
\lbrack U^{\delta }(x)(V_{n}-\beta )]^{\gamma })}{e^{\alpha \lbrack
U^{\delta }(x)(V_{n}-\beta )]^{\gamma }}}f(x)dx,\ n\in \mathbf{N},
\end{equation*}%
we can rewrite (\ref{3.2}) as $$||c||_{p,\Psi _{\beta }^{1-p}}<k(\sigma
)||f||_{p,\Phi _{\delta }}<\infty ,$$ namely, $c\in l_{p,\Psi _{\beta
}^{1-p}}.$

\begin{definition}
  Define a half-discrete Hardy-Hilbert-type operator $%
T_{1}:L_{p,\Phi _{\delta }}(\mathbf{R}_{+})\rightarrow l_{p,\Psi _{\beta
}^{1-p}}$ as follows: For any $f\in L_{p,\Phi _{\delta }}(\mathbf{R}_{+}),$
there exists a unique representation $T_{1}f=c\in l_{p,\Psi _{\beta }^{1-p}}.
$ Define the formal inner product of $T_{1}f$ and $a=\{a_{n}\}_{n=1}^{\infty
}\in l_{q,\Psi _{\beta }}$ as follows:%
\begin{equation}
(T_{1}f,a):=\sum_{n=1}^{\infty }\left[ \int_{0}^{\infty }\frac{\csc h(\rho
U^{\delta \gamma }(x)(V_{n}-\beta )^{\gamma })}{e^{\alpha U^{\delta \gamma
}(x)(V_{n}-\beta )^{\gamma }}}f(x)dx\right] a_{n}.  \label{3.14}
\end{equation}
\end{definition}

\noindent Then we can rewrite (\ref{3.1}) and (\ref{3.2}) as follows:%
\begin{eqnarray}
(T_{1}f,a) &<&k(\sigma )||f||_{p,\Phi _{\delta }}||a||_{q,\Psi _{\beta }},
\label{3.15} \\
||T_{1}f||_{p,\Psi _{\beta }^{1-p}} &<&k(\sigma )||f||_{p,\Phi _{\delta }}.
\label{3.16}
\end{eqnarray}

\noindent Define the norm of operator $T_{1}$ as follows:%
\begin{equation*}
||T_{1}||:=\sup_{f(\neq \theta )\in L_{p,\Phi _{\delta }}(\mathbf{R}_{+})}%
\frac{||T_{1}f||_{p,\Psi _{\beta }^{1-p}}}{||f||_{p,\Phi _{\delta }}}.
\end{equation*}%
Then by (\ref{3.16}), it follows that $||T_{1}||\leq k(\sigma ).$ Since by
Theorem 3.2, the constant factor in (\ref{3.16}) is the best possible, we have%
\begin{equation}
||T_{1}||=k(\sigma )=\frac{2\Gamma (\frac{\sigma }{\gamma })}{\gamma (2\rho
)^{\sigma /\gamma }}\zeta (\frac{\sigma }{\gamma },\frac{\alpha +\rho }{%
2\rho }).  \label{3.17}
\end{equation}

\noindent Assuming that $a=\{a_{n}\}_{n=1}^{\infty }\in l_{q,\Psi _{\beta }}$ and setting
\begin{equation*}
h(x):=\sum_{n=1}^{\infty }\frac{\csc h(\rho U^{\delta \gamma
}(x)(V_{n}-\beta )^{\gamma })}{e^{\alpha U^{\delta \gamma }(x)(V_{n}-\beta
)^{\gamma }}}a_{n},\ x\in \mathbf{R}_{+},
\end{equation*}%
we can rewrite (\ref{3.3}) as $||h||_{q,\Phi _{\delta }^{1-q}}<k(\sigma
)||a||_{q,\Psi _{\beta }}<\infty ,$ namely, $h\in L_{q,\Phi _{\delta
}^{1-q}}(\mathbf{R}_{+}).$

\begin{definition}
  Define a half-discrete Hardy-Hilbert-type operator $%
T_{2}:l_{q,\Psi _{\beta }}\rightarrow L_{q,\Phi _{\delta }^{1-q}}(\mathbf{R}%
_{+})$ as follows: For any $a=\{a_{n}\}_{n=1}^{\infty }\in l_{q,\Psi _{\beta
}},$ there exists a unique representation $$T_{2}a=h\in L_{q,\Phi _{\delta
}^{1-q}}(\mathbf{R}_{+}).$$ 
Define the formal inner product of $T_{2}a$ and $%
f\in L_{p,\Phi _{\delta }}(\mathbf{R}_{+})$ as follows:%
\begin{equation}
(T_{2}a,f):=\int_{0}^{\infty }\left[ \sum_{n=1}^{\infty }\frac{\csc h(\rho
U^{\delta \gamma }(x)(V_{n}-\beta )^{\gamma })}{e^{\alpha U^{\delta \gamma
}(x)(V_{n}-\beta )^{\gamma }}}a_{n}\right] f(x)dx.  \label{3.18}
\end{equation}
\end{definition}

Then we can rewrite (\ref{3.1}) and (\ref{3.3}) as follows:%
\begin{eqnarray}
(T_{2}a,f) &<&k(\sigma )||f||_{p,\Phi _{\delta }}||a||_{q,\Psi _{\beta }},
\label{3.19} \\
||T_{2}a||_{q,\Phi _{\delta }^{1-q}} &<&k(\sigma )||a||_{q,\Psi _{\beta }}.
\label{3.20}
\end{eqnarray}

Define the norm of operator $T_{2}$ as follows:%
\begin{equation*}
||T_{2}||:=\sup_{a(\neq \theta )\in l_{q,\Psi }}\frac{||T_{2}a||_{q,\Phi
_{\delta }^{1-q}}}{||a||_{q,\Psi _{\beta }}}.
\end{equation*}%
Then by (\ref{3.20}), we find $||T_{2}||\leq k(\sigma ).$ Since by Theorem
3.2, the constant factor in (\ref{3.20}) is the best possible, we obtain%
\begin{equation}
||T_{2}||=k(\sigma )=\frac{2\Gamma (\frac{\sigma }{\gamma })}{\gamma (2\rho
)^{\sigma /\gamma }}\zeta (\frac{\sigma }{\gamma },\frac{\alpha +\rho }{%
2\rho })=||T_{1}||.  \label{3.21}
\end{equation}

\section{Some Equivalent Reverse Inequalities}

In the following, we also set
\begin{equation*}
\widetilde{\Phi }_{\delta }(x):=(1-\theta _{\delta }(\sigma ,x))\frac{%
U^{p(1-\delta \sigma )-1}(x)}{\mu ^{p-1}(x)}\ (x\in \mathbf{R}_{+}).
\end{equation*}%
For $0<p<1$ or $p<0,$ we still use the formal symbols $||f||_{p,\Phi
_{\delta }}$, $||f||_{p,\widetilde{\Phi }_{\delta }}$ and $||a||_{q,\Psi_{\beta} }$ et al.

\begin{theorem}
  If $\rho >\max \{0,-\alpha \},0<\gamma <\sigma \leq
1, k(\sigma )$ is indicated by (\ref{2.1}), and
\mbox{$U(\infty )=V(\infty )=\infty ,$} then for $p<0,$ $0<||f||_{p,\Phi _{\delta
}},||a||_{q,\Psi _{\beta }}<\infty,$ we have the following equivalent
inequalities with the best possible constant factor $k(\sigma )$:%
\begin{eqnarray}
I &=&\sum_{n=1}^{\infty }\int_{0}^{\infty }\frac{\csc h(\rho U^{\delta
\gamma }(x)(V_{n}-\beta )^{\gamma })}{e^{\alpha U^{\delta \gamma
}(x)(V_{n}-\beta )^{\gamma }}}a_{n}f(x)dx>k(\sigma )||f||_{p,\Phi _{\delta
}}||a||_{q,\Psi _{\beta }},  \label{4.1} \\
J_{1} &=&\sum_{n=1}^{\infty }\frac{\nu _{n+1}}{(V_{n}-\beta )^{1-p\sigma }}%
\left[ \int_{0}^{\infty }\frac{\csc h(\rho U^{\delta \gamma }(x)(V_{n}-\beta
)^{\gamma })}{e^{\alpha U^{\delta \gamma }(x)(V_{n}-\beta )^{\gamma }}}f(x)dx%
\right] ^{p}>k(\sigma )||f||_{p,\Phi _{\delta }},  \label{4.2} \\
J_{2} &=&\left\{ \int_{0}^{\infty }\frac{\mu (x)}{U^{1-q\delta \sigma }(x)}%
\left[ \sum_{n=1}^{\infty }\frac{\csc h(\rho U^{\delta \gamma
}(x)(V_{n}-\beta )^{\gamma })}{e^{\alpha U^{\delta \gamma }(x)(V_{n}-\beta
)^{\gamma }}}a_{n}\right] ^{q}dx\right\} ^{\frac{1}{q}}  \notag \\
&>&k(\sigma )||a||_{q,\Psi _{\beta }}.  \label{4.3}
\end{eqnarray}
\end{theorem}

\begin{proof}
 By the reverse weighted H\"{o}lder inequality (cf. \cite%
{K1}), since $p<0,$ similarly to the way we obtained (\ref{3.4}) and (\ref%
{3.5}), we have%
\begin{eqnarray*}
&&\left[ \int_{0}^{\infty }\frac{\csc h(\rho U^{\delta \gamma
}(x)(V_{n}-\beta )^{\gamma })}{e^{\alpha U^{\delta \gamma }(x)(V_{n}-\beta
)^{\gamma }}}f(x)dx\right] ^{p} \\
&\leq &\frac{(\varpi _{\delta }(\sigma ,n))^{p-1}}{(V_{n}-\beta )^{p\sigma
-1}\nu _{n+1}}\int_{0}^{\infty }\frac{\csc h(\rho U^{\delta \gamma
}(x)(V_{n}-\beta )^{\gamma })}{e^{\alpha U^{\delta \gamma }(x)(V_{n}-\beta
)^{\gamma }}}\frac{U^{(1-\delta \sigma )(p-1)}(x)\nu _{n+1}}{(V_{n}-\beta
)^{1-\sigma }\mu ^{p-1}(x)}f^{p}(x)dx.
\end{eqnarray*}%
Then by (\ref{2.9}) and the Lebesgue term by term integration theorem, it
follows that%
\begin{eqnarray*}
J_{1} &\geq &(k(\sigma ))^{\frac{1}{q}}\left[ \sum_{n=1}^{\infty
}\int_{0}^{\infty }\frac{\csc h(\rho U^{\delta \gamma }(x)(V_{n}-\beta
)^{\gamma })}{e^{\alpha U^{\delta \gamma }(x)(V_{n}-\beta )^{\gamma }}}\frac{%
U^{(1-\delta \sigma )(p-1)}(x)\nu _{n+1}}{(V_{n}-\beta )^{1-\sigma }\mu
^{p-1}(x)}f^{p}(x)dx\right] ^{\frac{1}{p}} \\
&=&(k(\sigma ))^{\frac{1}{q}}\left[ \int_{0}^{\infty }\omega _{\delta
}(\sigma ,x)\frac{U^{p(1-\delta \sigma )-1}(x)}{\mu ^{p-1}(x)}f^{p}(x)dx%
\right] ^{\frac{1}{p}}.
\end{eqnarray*}%
Then by (\ref{2.7}), we have (\ref{4.2}).

By the reverse H\"{o}lder inequality (cf. \cite{K1}), we have%
\begin{eqnarray}
I &=&\sum_{n=1}^{\infty }\left[ \frac{\nu _{n+1}^{\frac{1}{p}}}{(V_{n}-\beta
)^{\frac{1}{p}-\sigma }}\int_{0}^{\infty }\frac{\csc h(\rho U^{\delta \gamma
}(x)(V_{n}-\beta )^{\gamma })}{e^{\alpha U^{\delta \gamma }(x)(V_{n}-\beta
)^{\gamma }}}f(x)dx\right] \left[ \frac{(V_{n}-\beta )^{\frac{1}{p}-\sigma
}a_{n}}{\nu _{n+1}^{\frac{1}{p}}}\right]   \notag \\
&\geq &J_{1}||a||_{q,\Psi _{\beta }}.  \label{4.4}
\end{eqnarray}%
Then by (\ref{4.2}), we derive (\ref{4.1}). On the other hand, assuming that (%
\ref{4.1}) is valid, we set $a_{n}$ as in Theorem 3.1. Then we obtain $$%
J_{1}^{p}=||a||_{q,\Psi _{\beta }}^{q}.$$ If $J_{1}=\infty ,$ then (\ref{4.2}%
) is trivially valid. If $J_{1}=0,$ then (\ref{4.2}) is still not valid.
Suppose that $0<J_{1}<\infty .$ By (\ref{4.1}), it follows that%
\begin{eqnarray*}
||a||_{q,\Psi _{\beta }}^{q} &=&J_{1}^{p}=I>k(\sigma )||f||_{p,\Phi _{\delta
}}||a||_{q,\Psi _{\beta }}, \\
||a||_{q,\Psi _{\beta }}^{q-1} &=&J_{1}>k(\sigma )||f||_{p,\Phi _{\delta }},
\end{eqnarray*}%
and then (\ref{4.2}) follows, which is equivalent to (\ref{4.1}).

Applying again the weighted reverse H\"{o}lder inequality (cf. \cite{K1}),
since $0<q<1,$ similarly to how we obtained (\ref{3.7}) and (\ref{3.8}),
we have%
\begin{eqnarray*}
&&\left[ \sum_{n=1}^{\infty }\frac{\csc h(\rho U^{\delta \gamma
}(x)(V_{n}-\beta )^{\gamma })}{e^{\alpha U^{\delta \gamma }(x)(V_{n}-\beta
)^{\gamma }}}a_{n}\right] ^{q} \\
&\geq &\frac{(\omega _{\delta }(\sigma ,x))^{q-1}}{U^{q\delta \sigma
-1}(x)\mu (x)}\sum_{n=1}^{\infty }\frac{\csc h(\rho U^{\delta \gamma
}(x)(V_{n}-\beta )^{\gamma })}{e^{\alpha U^{\delta \gamma }(x)(V_{n}-\beta
)^{\gamma }}}\frac{(V_{n}-\beta )^{(1-\sigma )(q-1)}\mu (x)}{U^{1-\delta
\sigma }(x)\nu _{n+1}^{q-1}}a_{n}^{q}.
\end{eqnarray*}%
Then, by (\ref{2.7}) and the Lebesgue term by term integration theorem, it
follows that%
\begin{eqnarray*}
J_{2} &>&(k(\sigma ))^{\frac{1}{p}}\left[ \int_{0}^{\infty
}\sum_{n=1}^{\infty }\frac{\csc h(\rho U^{\delta \gamma }(x)(V_{n}-\beta
)^{\gamma })}{e^{\alpha U^{\delta \gamma }(x)(V_{n}-\beta )^{\gamma }}}\frac{%
(V_{n}-\beta )^{(1-\sigma )(q-1)}\mu (x)}{U^{1-\delta \sigma }(x)\nu
_{n+1}^{q-1}}a_{n}^{q}dx\right] ^{\frac{1}{q}} \\
&=&(k(\sigma ))^{\frac{1}{p}}\left[ \sum_{n=1}^{\infty }\varpi _{\delta
}(\sigma ,n)\frac{(V_{n}-\beta )^{q(1-\sigma )-1}}{\nu _{n+1}^{q-1}}a_{n}^{q}%
\right] ^{\frac{1}{q}}.
\end{eqnarray*}%
Hence, by (\ref{2.9}), we have (\ref{4.3}).

By the reverse H\"{o}lder inequality (cf. \cite{K1}), we get%
\begin{eqnarray}
I &=&\int_{0}^{\infty }\left( \frac{U^{\frac{1}{q}-\delta \sigma }(x)}{\mu ^{%
\frac{1}{q}}(x)}f(x)\right) \left[ \frac{\mu ^{\frac{1}{q}}(x)}{U^{\frac{1}{q%
}-\delta \sigma }(x)}\sum_{n=1}^{\infty }\frac{\csc h(\rho U^{\delta \gamma
}(x)(V_{n}-\beta )^{\gamma })}{e^{\alpha U^{\delta \gamma }(x)(V_{n}-\beta
)^{\gamma }}}a_{n}\right] dx  \notag \\
&\geq &||f||_{p,\Phi _{\delta }}J_{2}.  \label{4.5}
\end{eqnarray}%
Thus by (\ref{4.3}), we obtain (\ref{4.1}). On the other hand, assuming that (%
\ref{4.3}) is valid, we set $f(x)$ as in Theorem 4.1. Then we derive that $$%
J_{2}^{q}=||f||_{p,\Phi _{\delta }}^{p}.$$ 
If $J_{2}=\infty ,$ then (\ref{4.3}%
) is trivially valid. If $J_{2}=0,$ then (\ref{4.3}) remains impossible.
Suppose that $0<J_{2}<\infty .$ By (\ref{4.1}), it follows that%
\begin{eqnarray*}
||f||_{p,\Phi _{\delta }}^{p} &=&J_{2}^{q}=I>k(\sigma )||f||_{p,\Phi
_{\delta }}||a||_{q,\Psi _{\beta }}, \\
||f||_{p,\Phi _{\delta }}^{p-1} &=&J_{2}>k(\sigma )||a||_{q,\Psi _{\beta }},
\end{eqnarray*}%
and then (\ref{4.3}) follows, which is equivalent to (\ref{4.1}).

\noindent Therefore, inequalities (\ref{4.1}), (\ref{4.2}) and (\ref{4.3}) are
equivalent.

\noindent For $\varepsilon \in (0,\frac{q(\sigma -\gamma )}{2}),$ we set $\widetilde{%
\sigma }=\sigma -\frac{\varepsilon }{q},$ and $\widetilde{f}=\widetilde{f}%
(x),\ x\in \mathbf{R}_{+},\ \widetilde{a}=\{\widetilde{a}_{n}\}_{n=1}^{\infty },$%
\begin{eqnarray*}
\widetilde{f}(x) &=&\left\{
\begin{array}{c}
U^{\delta (\widetilde{\sigma }+\varepsilon )-1}(x)\mu (x),0<x^{\delta }\leq 1
\\
0,x^{\delta }>0%
\end{array}%
\right. , \\
\widetilde{a}_{n} &=&(V_{n}-\beta )^{\widetilde{\sigma }-1}\nu
_{n+1}=(V_{n}-\beta )^{\sigma -\frac{\varepsilon }{q}-1}\nu _{n+1},\ n\in
\mathbf{N}.
\end{eqnarray*}%
By (\ref{2.11}), (\ref{3.12}) and (\ref{2.7}), we obtain%
\begin{equation}
||\widetilde{f}||_{p,\Phi _{\delta }}||\widetilde{a}||_{q,\Psi_{\beta} }=\frac{1}{%
\varepsilon }U^{\frac{\delta \varepsilon }{p}}(1)\left[ \frac{1}{(\nu
_{1}-\beta )^{\varepsilon }}+\varepsilon O(1)\right] ^{\frac{1}{q}},  \notag
\end{equation}%
\begin{eqnarray*}
\widetilde{I} &=&\sum_{n=1}^{\infty }\int_{0}^{\infty }\frac{\csc h(\rho
U^{\delta \gamma }(x)(V_{n}-\beta )^{\gamma })}{e^{\alpha U^{\delta \gamma
}(x)(V_{n}-\beta )^{\gamma }}}\widetilde{a}_{n}\widetilde{f}(x)dx \\
&=&\int_{\{x>0;0<x^{\delta }\leq 1\}}\omega _{\delta }(\widetilde{\sigma },x)%
\frac{\mu (x)}{U^{1-\delta \varepsilon }(x)}dx \\
&\leq &k(\widetilde{\sigma })\int_{\{x>0;0<x^{\delta }\leq 1\}}\frac{\mu (x)%
}{U^{1-\delta \varepsilon }(x)}dx \\
&=&\frac{1}{\varepsilon }k(\sigma -\frac{\varepsilon }{q})U^{\delta
\varepsilon }(1).
\end{eqnarray*}

If there exists a positive constant $K\geq k(\sigma ),$ such that (\ref{4.1}%
) is valid when replacing $k(\sigma )$ to $K,$ then in particular, we have 
$$%
\varepsilon \widetilde{I}>\varepsilon K||\widetilde{f}||_{p,\Phi _{\delta
}}||\widetilde{a}||_{q,\Psi \beta },$$ namely,%
\begin{equation*}
k(\sigma -\frac{\varepsilon }{q})U^{\delta \varepsilon }(1)>K\cdot U^{\frac{%
\delta \varepsilon }{p}}(1)\left[ \frac{1}{(\nu _{1}-\beta )^{\varepsilon }}%
+\varepsilon O(1)\right] ^{\frac{1}{q}}.
\end{equation*}%
It follows that $k(\sigma )\geq K(\varepsilon \rightarrow 0^{+}).$ Hence, $%
K=k(\sigma )$ is the best possible constant factor of (\ref{4.1}).

The constant factor $k(\sigma )$ in (\ref{4.2}) ((\ref{4.3})) is still the
best possible. Otherwise, we would reach a contradiction by (\ref{4.4}) ((%
\ref{4.5})) that the constant factor in (\ref{4.1}) is not the best
possible.
\end{proof}

\begin{theorem}
With the assumptions of Theorem 4.1, if $$0<p<1,\
0<||f||_{p,\Phi _{\delta }},||a||_{q,\Psi _{\beta }}<\infty ,$$ then we have
the following equivalent inequalities with the best possible constant factor
$k(\sigma )$:%
\begin{eqnarray}
I &=&\sum_{n=1}^{\infty }\int_{0}^{\infty }\frac{\csc h(\rho U^{\delta
\gamma }(x)(V_{n}-\beta )^{\gamma })}{e^{\alpha U^{\delta \gamma
}(x)(V_{n}-\beta )^{\gamma }}}a_{n}f(x)dx>k(\sigma )||f||_{p,\widetilde{\Phi
}_{\delta }}||a||_{q,\Psi _{\beta }},  \label{4.6} \\
J_{1} &=&\sum_{n=1}^{\infty }\frac{\nu _{n+1}}{(V_{n}-\beta )^{1-p\sigma }}%
\left[ \int_{0}^{\infty }\frac{\csc h(\rho U^{\delta \gamma }(x)(V_{n}-\beta
)^{\gamma })}{e^{\alpha U^{\delta \gamma }(x)(V_{n}-\beta )^{\gamma }}}f(x)dx%
\right] ^{p}>k(\sigma )||f||_{p,\widetilde{\Phi }_{\delta }},  \label{4.7} \\
J &:=&\left\{ \int_{0}^{\infty }\frac{(1-\theta _{\delta }(\sigma
,x))^{1-q}\mu (x)}{U^{1-q\delta \sigma }(x)}\left[ \sum_{n=1}^{\infty }\frac{%
\csc h(\rho U^{\delta \gamma }(x)(V_{n}-\beta )^{\gamma })}{e^{\alpha
U^{\delta \gamma }(x)(V_{n}-\beta )^{\gamma }}}a_{n}\right] ^{q}dx\right\} ^{%
\frac{1}{q}}  \notag \\
&>&k(\sigma )||a||_{q,\Psi _{\beta }}.  \label{4.8}
\end{eqnarray}
\end{theorem}

\begin{proof}
 By the reverse weighted H\"{o}lder inequality (cf. \cite%
{K1}), since $0<p<1,$ similarly to as we obtained (\ref{3.4}) and (\ref%
{3.5}), we have%
 \begin{eqnarray*}
&&\left[ \int_{0}^{\infty }\frac{\csc h(\rho U^{\delta \gamma
}(x)(V_{n}-\beta )^{\gamma })}{e^{\alpha U^{\delta \gamma }(x)(V_{n}-\beta
)^{\gamma }}}f(x)dx\right] ^{p} \\
&\geq &\frac{(\varpi _{\delta }(\sigma ,n))^{p-1}}{(V_{n}-\beta )^{p\sigma
-1}\nu _{n+1}}\int_{0}^{\infty }\frac{\csc h(\rho U^{\delta \gamma
}(x)(V_{n}-\beta )^{\gamma })}{e^{\alpha U^{\delta \gamma }(x)(V_{n}-\beta
)^{\gamma }}}\frac{U^{(1-\delta \sigma )(p-1)}(x)\nu _{n+1}}{(V_{n}-\beta
)^{1-\sigma }\mu ^{p-1}(x)}f^{p}(x)dx.
\end{eqnarray*}%
In view of (\ref{2.9}) and the Lebesgue term by term integration theorem, we find%
\begin{eqnarray*}
J_{1} &\geq &(k(\sigma ))^{\frac{1}{q}}\left[ \sum_{n=1}^{\infty
}\int_{0}^{\infty }\frac{\csc h(\rho U^{\delta \gamma }(x)(V_{n}-\beta
)^{\gamma })}{e^{\alpha U^{\delta \gamma }(x)(V_{n}-\beta )^{\gamma }}}\frac{%
U^{(1-\delta \sigma )(p-1)}(x)\nu _{n+1}}{(V_{n}-\beta )^{1-\sigma }\mu
^{p-1}(x)}f^{p}(x)dx\right] ^{\frac{1}{p}} \\
&=&(k(\sigma ))^{\frac{1}{q}}\left[ \int_{0}^{\infty }\omega _{\delta
}(\sigma ,x)\frac{U^{p(1-\delta \sigma )-1}(x)}{\mu ^{p-1}(x)}f^{p}(x)dx%
\right] ^{\frac{1}{p}}.
\end{eqnarray*}%
Then by (\ref{2.10}), we have (\ref{4.7}).

By the reverse H\"{o}lder inequality (cf. \cite{K1}), we have%
\begin{eqnarray}
I &=&\sum_{n=1}^{\infty }\left[ \frac{\nu _{n+1}^{\frac{1}{p}}}{(V_{n}-\beta
)^{\frac{1}{p}-\sigma }}\int_{0}^{\infty }\frac{\csc h(\rho U^{\delta \gamma
}(x)(V_{n}-\beta )^{\gamma })}{e^{\alpha U^{\delta \gamma }(x)(V_{n}-\beta
)^{\gamma }}}f(x)dx\right] \left[ \frac{(V_{n}-\beta )^{\frac{1}{p}-\sigma
}a_{n}}{\nu _{n+1}^{\frac{1}{p}}}\right]   \notag \\
&\geq &J_{1}||a||_{q,\Psi _{\beta }}.  \label{4.9}
\end{eqnarray}%
Then by (\ref{4.7}), we have (\ref{4.6}). On the other hand, assuming that (%
\ref{4.6}) is valid, we set $a_{n}$ as in Theorem 3.1. Then we find $%
J_{1}^{p}=||a||_{q,\Psi _{\beta }}^{q}.$ If $J_{1}=\infty ,$ then (\ref{4.7}%
) is trivially valid; if $J_{1}=0,$ then (\ref{4.7}) keeps impossible.
Suppose that $0<J_{1}<\infty .$ By (\ref{4.6}), it follows that%
\begin{eqnarray*}
||a||_{q,\Psi }^{q} &=&J_{1}^{p}=I>k(\sigma )||f||_{p,\widetilde{\Phi }%
_{\delta }}||a||_{q,\Psi _{\beta }}, \\
||a||_{q,\Psi }^{q-1} &=&J_{1}>k(\sigma )||f||_{p,\widetilde{\Phi }_{\delta
}},
\end{eqnarray*}%
and then (\ref{4.7}) follows, which is equivalent to (\ref{4.6}).

\noindent Similarly, by the reverse weighted H\"{o}lder inequality (cf. \cite{K1}),
since $q<0,$ we have%
\begin{eqnarray*}
&&\left[ \sum_{n=1}^{\infty }\frac{\csc h(\rho U^{\delta \gamma
}(x)(V_{n}-\beta )^{\gamma })}{e^{\alpha U^{\delta \gamma }(x)(V_{n}-\beta
)^{\gamma }}}a_{n}\right] ^{q} \\
&\leq &\frac{(\omega _{\delta }(\sigma ,x))^{q-1}}{U^{q\delta \sigma
-1}(x)\mu (x)}\sum_{n=1}^{\infty }\frac{\csc h(\rho U^{\delta \gamma
}(x)(V_{n}-\beta )^{\gamma })}{e^{\alpha U^{\delta \gamma }(x)(V_{n}-\beta
)^{\gamma }}}\frac{(V_{n}-\beta )^{(1-\sigma )(q-1)}\mu (x)}{U^{1-\delta
\sigma }(x)\nu _{n+1}^{q-1}}a_{n}^{q}.
\end{eqnarray*}%
Therefore, by (\ref{2.10}) and the Lebesgue term by term integration theorem, it
follows that%
\begin{eqnarray*}
J &>&(k(\sigma ))^{\frac{1}{p}}\left[ \int_{0}^{\infty }\sum_{n=1}^{\infty }%
\frac{\csc h(\rho U^{\delta \gamma }(x)(V_{n}-\beta )^{\gamma })}{e^{\alpha
U^{\delta \gamma }(x)(V_{n}-\beta )^{\gamma }}}\frac{(V_{n}-\beta
)^{(1-\sigma )(q-1)}\mu (x)}{U^{1-\delta \sigma }(x)\nu _{n+1}^{q-1}}%
a_{n}^{q}dx\right] ^{\frac{1}{q}} \\
&=&(k(\sigma ))^{\frac{1}{p}}\left[ \sum_{n=1}^{\infty }\varpi _{\delta
}(\sigma ,n)\frac{(V_{n}-\beta )^{q(1-\sigma )-1}}{\nu _{n+1}^{q-1}}a_{n}^{q}%
\right] ^{\frac{1}{q}}.
\end{eqnarray*}%
Hence, by (\ref{2.9}), \ we have (\ref{4.8}).

\noindent By the reverse H\"{o}lder inequality (cf. \cite{K1}), we have%
\begin{equation*}
I=\int_{0}^{\infty }\left[ (1-\theta _{\delta }(\sigma ,x))^{\frac{1}{p}}%
\frac{U^{\frac{1}{q}-\delta \sigma }(x)}{\mu ^{\frac{1}{q}}(x)}f(x)\right]
\end{equation*}%
\begin{equation}
\times \left[ \frac{(1-\theta _{\delta }(\sigma ,x))^{\frac{-1}{p}}\mu ^{%
\frac{1}{q}}(x)}{U^{\frac{1}{q}-\delta \sigma }(x)}\sum_{n=1}^{\infty }\frac{%
\csc h(\rho U^{\delta \gamma }(x)(V_{n}-\beta )^{\gamma })}{e^{\alpha
U^{\delta \gamma }(x)(V_{n}-\beta )^{\gamma }}}a_{n}\right] dx\geq ||f||_{p,%
\widetilde{\Phi }_{\delta }}J.  \label{4.10}
\end{equation}%
Then by (\ref{4.8}), we have (\ref{4.6}). On the other hand, assuming that (%
\ref{4.6}) is valid, we set $f(x)$ as in Theorem 3.1. Then we derive that $%
J^{q}=||f||_{p,\widetilde{\Phi }_{\delta }}^{p}.$ If $J=\infty ,$ then (\ref%
{4.8}) is trivially valid; if $J=0,$ then (\ref{4.8}) is still not valid.
Suppose that $0<J<\infty .$ By (\ref{4.6}), it follows that%
\begin{eqnarray*}
||f||_{p,\widetilde{\Phi }_{\delta }}^{p} &=&J^{q}=I>k(\sigma )||f||_{p,%
\widetilde{\Phi }_{\delta }}||a||_{q,\Psi _{\beta }}, \\
||f||_{p,\widetilde{\Phi }_{\delta }}^{p-1} &=&J>k(\sigma )||a||_{q,\Psi
_{\beta }},
\end{eqnarray*}%
and then (\ref{4.8}) follows, which is equivalent to (\ref{4.6}).

\noindent Therefore, inequalities (\ref{4.6}), (\ref{4.7}) and (\ref{4.8}) are
equivalent.

\noindent For $\varepsilon \in (0,\frac{p(\sigma -\gamma )}{2}),$ we set $\widetilde{%
\sigma }=\sigma +\frac{\varepsilon }{p},$ and $\widetilde{f}=\widetilde{f}%
(x),x\in \mathbf{R}_{+},\widetilde{a}=\{\widetilde{a}_{n}\}_{n=1}^{\infty },$%
\begin{eqnarray*}
\widetilde{f}(x) &=&\left\{
\begin{array}{c}
U^{\delta \widetilde{\sigma }-1}(x)\mu (x),0<x^{\delta }\leq 1 \\
0,x^{\delta }>0%
\end{array}%
\right. , \\
\widetilde{a}_{n} &=&(V_{n}-\beta )^{\widetilde{\sigma }-\varepsilon -1}\nu
_{n+1}=(V_{n}-\beta )^{\sigma -\frac{\varepsilon }{q}-1}\nu _{n+1},n\in
\mathbf{N}.
\end{eqnarray*}%
By (\ref{2.10}), (\ref{2.11}) and (\ref{3.12}), we obtain%
\begin{eqnarray*}
&&||\widetilde{f}||_{p,\widetilde{\Phi }_{\delta }}||\widetilde{a}||_{q,\Psi_{\beta}
} \\
&=&\left[ \int_{\{x>0;0<x^{\delta }\leq 1\}}(1-O((U(x))^{\frac{\delta }{2}%
(\sigma -\gamma )}))\frac{\mu (x)dx}{U^{1-\delta \varepsilon }(x)}\right] ^{%
\frac{1}{p}}\left[ \sum_{n=1}^{\infty }\frac{\nu _{n+1}}{(V_{n}-\beta
)^{1+\varepsilon }}\right] ^{\frac{1}{q}} \\
&=&\frac{1}{\varepsilon }\left( U^{\delta \varepsilon }(1)-\varepsilon
O_{1}(1)\right) ^{\frac{1}{p}}\left[ \frac{1}{(\nu _{1}-\beta )^{\varepsilon
}}+\varepsilon O(1)\right] ^{\frac{1}{q}},
\end{eqnarray*}%
\begin{eqnarray*}
\widetilde{I} &=&\sum_{n=1}^{\infty }\int_{0}^{\infty }\frac{\csc h(\rho
U^{\delta \gamma }(x)(V_{n}-\beta )^{\gamma })}{e^{\alpha U^{\delta \gamma
}(x)(V_{n}-\beta )^{\gamma }}}\widetilde{a}_{n}\widetilde{f}(x)dx \\
&=&\sum_{n=1}^{\infty }\left[ \int_{\{x>0;0<x^{\delta }\leq 1\}}\frac{\csc
h(\rho U^{\delta \gamma }(x)(V_{n}-\beta )^{\gamma })}{e^{\alpha U^{\delta
\gamma }(x)(V_{n}-\beta )^{\gamma }}}\frac{(V_{n}-\beta )^{\widetilde{\sigma
}}\mu (x)}{U^{1-\delta \widetilde{\sigma }}(x)}dx\right] \frac{\nu _{n+1}}{%
(V_{n}-\beta )^{1+\varepsilon }} \\
&\leq &\sum_{n=1}^{\infty }\left[ \int_{0}^{\infty }\frac{\csc h(\rho
U^{\delta \gamma }(x)(V_{n}-\beta )^{\gamma })}{e^{\alpha U^{\delta \gamma
}(x)(V_{n}-\beta )^{\gamma }}}\frac{(V_{n}-\beta )^{\widetilde{\sigma }}\mu
(x)}{U^{1-\delta \widetilde{\sigma }}(x)}dx\right] \frac{\nu _{n+1}}{%
(V_{n}-\beta )^{1+\varepsilon }} \\
&=&\sum_{n=1}^{\infty }\varpi _{\delta }(\widetilde{\sigma },n)\frac{\nu
_{n+1}}{(V_{n}-\beta )^{1+\varepsilon }}=k(\widetilde{\sigma }%
)\sum_{n=1}^{\infty }\frac{\nu _{n+1}}{(V_{n}-\beta )^{1+\varepsilon }} \\
&=&\frac{1}{\varepsilon }k(\sigma +\frac{\varepsilon }{p})\left[ \frac{1}{%
(\nu _{1}-\beta )^{\varepsilon }}+\varepsilon O(1)\right] .
\end{eqnarray*}

If there exists a positive constant $K\geq k(\sigma ),$ such that (\ref{4.1}%
) is valid when replacing $k(\sigma )$ by $K,$ then in particular, we have $$%
\varepsilon \widetilde{I}>\varepsilon K||\widetilde{f}||_{p,\widetilde{\Phi }%
_{\delta }}||\widetilde{a}||_{q,\Psi _{\beta }},$$ namely,%
\begin{eqnarray*}
&&k(\sigma +\frac{\varepsilon }{p})\left[ \frac{1}{(\nu _{1}-\beta
)^{\varepsilon }}+\varepsilon O(1)\right]  \\
&>&K\left( U^{\delta \varepsilon }(1)-\varepsilon O_{1}(1)\right) ^{\frac{1}{%
p}}\left[ \frac{1}{(\nu _{1}-\beta )^{\varepsilon }}+\varepsilon O(1)\right]
^{\frac{1}{q}}.
\end{eqnarray*}%
It follows that $k(\sigma )\geq K(\varepsilon \rightarrow 0^{+}).$ Hence, $%
K=k(\sigma )$ is the best possible constant factor of (\ref{4.6}).

The constant factor $k(\sigma )$ in (\ref{4.7}) ((\ref{4.8})) is still the
best possible. Otherwise, we would reach a contradiction by (\ref{4.9}) ((%
\ref{4.10})) that the constant factor in (\ref{4.6}) is not the best
possible.
\end{proof}

\section{Some Corollaries}

For $\delta =1$ in Theorem 3.2, Theorem 4.1 and Theorem 4.2, the following inequalities with the
non-homogeneous kernel hold true:

\begin{corollary}
  If $\rho >\max \{0,-\alpha \},0<\gamma <\sigma \leq
1,k(\sigma )$ is indicated by (\ref{2.2}), and \mbox{$U(\infty )=V(\infty )=\infty
,$} then \\
(i) for $p>1,$ $0<||f||_{p,\Phi _{1}},||a||_{q,\Psi _{\beta
}}<\infty ,$ we have the following equivalent inequalities:

\begin{eqnarray}
&&\sum_{n=1}^{\infty }\int_{0}^{\infty }\frac{\csc h(\rho U^{\gamma
}(x)(V_{n}-\beta )^{\gamma })}{e^{\alpha U^{\gamma }(x)(V_{n}-\beta
)^{\gamma }}}a_{n}f(x)dx<k(\sigma )||f||_{p,\Phi _{1}}||a||_{q,\Psi _{\beta
}},  \label{5.1} \\
&&\sum_{n=1}^{\infty }\frac{\nu _{n+1}}{(V_{n}-\beta )^{1-p\sigma }}\left[
\int_{0}^{\infty }\frac{\csc h(\rho U^{\gamma }(x)(V_{n}-\beta )^{\gamma })}{%
e^{\alpha U^{\gamma }(x)(V_{n}-\beta )^{\gamma }}}f(x)dx\right]
^{p}<k(\sigma )||f||_{p,\Phi _{1}},  \label{5.2}
\end{eqnarray}%
\begin{equation}
\left\{ \int_{0}^{\infty }\frac{\mu (x)}{U^{1-q\sigma }(x)}\left[
\sum_{n=1}^{\infty }\frac{\csc h(\rho U^{\gamma }(x)(V_{n}-\beta )^{\gamma })%
}{e^{\alpha U^{\gamma }(x)(V_{n}-\beta )^{\gamma }}}a_{n}\right]
^{q}dx\right\} ^{\frac{1}{q}}<k(\sigma )||a||_{q,\Psi _{\beta }};
\label{5.3}
\end{equation}%
(ii) for $p<0,$ $0<||f||_{p,\Phi _{1}},||a||_{q,\Psi }<\infty ,$ we have the
following equivalent inequalities:%
\begin{equation}
\sum_{n=1}^{\infty }\int_{0}^{\infty }\frac{\csc h(\rho U^{\gamma
}(x)(V_{n}-\beta )^{\gamma })}{e^{\alpha U^{\gamma }(x)(V_{n}-\beta
)^{\gamma }}}a_{n}f(x)dx>k(\sigma )||f||_{p,\Phi _{1}}||a||_{q,\Psi _{\beta
}},  \label{5.4}
\end{equation}%
\begin{equation}
\sum_{n=1}^{\infty }\frac{\nu _{n+1}}{(V_{n}-\beta )^{1-p\sigma }}\left[
\int_{0}^{\infty }\frac{\csc h(\rho U^{\gamma }(x)(V_{n}-\beta )^{\gamma })}{%
e^{\alpha U^{\gamma }(x)(V_{n}-\beta )^{\gamma }}}f(x)dx\right]
^{p}>k(\sigma )||f||_{p,\Phi _{1}},  \label{5.5}
\end{equation}%
\begin{equation}
\left\{ \int_{0}^{\infty }\frac{\mu (x)}{U^{1-q\sigma }(x)}\left[
\sum_{n=1}^{\infty }\frac{\csc h(\rho U^{\gamma }(x)(V_{n}-\beta )^{\gamma })%
}{e^{\alpha U^{\gamma }(x)(V_{n}-\beta )^{\gamma }}}a_{n}\right]
^{q}dx\right\} ^{\frac{1}{q}}>k(\sigma )||a||_{q,\Psi _{\beta }};
\label{5.6}
\end{equation}%
(iii) for $0<p<1,$ $0<||f||_{p,\Phi _{1}},||a||_{q,\Psi }<\infty ,$ we have
the following equivalent inequalities:%
\begin{eqnarray}
&&\sum_{n=1}^{\infty }\int_{0}^{\infty }\frac{\csc h(\rho U^{\gamma
}(x)(V_{n}-\beta )^{\gamma })}{e^{\alpha U^{\gamma }(x)(V_{n}-\beta
)^{\gamma }}}a_{n}f(x)dx>k(\sigma )||f||_{p,\widetilde{\Phi }%
_{1}}||a||_{q,\Psi _{\beta }},  \label{5.7} \\
&&\sum_{n=1}^{\infty }\frac{\nu _{n+1}}{(V_{n}-\beta )^{1-p\sigma }}\left[
\int_{0}^{\infty }\frac{\csc h(\rho U^{\gamma }(x)(V_{n}-\beta )^{\gamma })}{%
e^{\alpha U^{\gamma }(x)(V_{n}-\beta )^{\gamma }}}f(x)dx\right]
^{p}>k(\sigma )||f||_{p,\widetilde{\Phi }_{1}},  \label{5.8}
\end{eqnarray}%
\begin{eqnarray}
&&\left\{ \int_{0}^{\infty }\frac{(1-\theta _{1}(\sigma ,x))^{1-q}\mu (x)}{%
U^{1-q\sigma }(x)}\left[ \sum_{n=1}^{\infty }\frac{\csc h(\rho U^{\gamma
}(x)(V_{n}-\beta )^{\gamma })}{e^{\alpha U^{\gamma }(x)(V_{n}-\beta
)^{\gamma }}}a_{n}\right] ^{q}dx\right\} ^{\frac{1}{q}}  \notag \\
&>&k(\sigma )||a||_{q,\Psi _{\beta }}.  \label{5.9}
\end{eqnarray}
\end{corollary}

The above inequalities involve the best possible constant factor $k(\sigma
).$

For $\delta =-1$ in  Theorem 3.2, Theorem 4.1 and Theorem 4.2, we have the following inequalities with the
homogeneous kernel of degree 0:

\begin{corollary}
  If $\rho >\max \{0,-\alpha \},0<\gamma <\sigma \leq
1,k(\sigma )$ is indicated by (\ref{2.2}), and \mbox{$U(\infty )=V(\infty )=\infty
,$} then\\
 (i) for $p>1,$ $0<||f||_{p,\Phi _{-1}},||a||_{q,\Psi _{\beta
}}<\infty ,$ we have the following equivalent inequalities:

\begin{eqnarray}
&&\sum_{n=1}^{\infty }\int_{0}^{\infty }\frac{\csc h(\rho (\frac{V_{n}-\beta
}{U(x)})^{\gamma })}{e^{\alpha (\frac{V_{n}}{U(x)})^{\gamma }}}%
a_{n}f(x)dx<k(\sigma )||f||_{p,\Phi _{-1}}||a||_{q,\Psi _{\beta }},
\label{5.10} \\
&&\sum_{n=1}^{\infty }\frac{\nu _{n+1}}{(V_{n}-\beta )^{1-p\sigma }}\left[
\int_{0}^{\infty }\frac{\csc h(\rho (\frac{V_{n}-\beta }{U(x)})^{\gamma })}{%
e^{\alpha (\frac{V_{n}-\beta }{U(x)})^{\gamma }}}f(x)dx\right] ^{p}<k(\sigma
)||f||_{p,\Phi _{-1}},  \label{5.11}
\end{eqnarray}%
\begin{equation}
\left\{ \int_{0}^{\infty }\frac{\mu (x)}{U^{1+q\sigma }(x)}\left[
\sum_{n=1}^{\infty }\frac{\csc h(\rho (\frac{V_{n}-\beta }{U(x)})^{\gamma })%
}{e^{\alpha (\frac{V_{n}-\beta }{U(x)})^{\gamma }}}a_{n}\right]
^{q}dx\right\} ^{\frac{1}{q}}<k(\sigma )||a||_{q,\Psi _{\beta }};
\label{5.12}
\end{equation}%
(ii) for $p<0,$ $0<||f||_{p,\Phi _{-1}},||a||_{q,\Psi _{\beta }}<\infty ,$
we have the following equivalent inequalities:%
\begin{eqnarray}
&&\sum_{n=1}^{\infty }\int_{0}^{\infty }\frac{\csc h(\rho (\frac{V_{n}-\beta
}{U(x)})^{\gamma })}{e^{\alpha (\frac{V_{n}-\beta }{U(x)})^{\gamma }}}%
a_{n}f(x)dx>k(\sigma )||f||_{p,\Phi _{-1}}||a||_{q,\Psi _{\beta }},
\label{5.13} \\
&&\sum_{n=1}^{\infty }\frac{\nu _{n+1}}{(V_{n}-\beta )^{1-p\sigma }}\left[
\int_{0}^{\infty }\frac{\csc h(\rho (\frac{V_{n}-\beta }{U(x)})^{\gamma })}{%
e^{\alpha (\frac{V_{n}-\beta }{U(x)})^{\gamma }}}f(x)dx\right] ^{p}>k(\sigma
)||f||_{p,\Phi _{-1}},  \label{5.14}
\end{eqnarray}%
\begin{equation}
\left\{ \int_{0}^{\infty }\frac{\mu (x)}{U^{1+q\sigma }(x)}\left[
\sum_{n=1}^{\infty }\frac{\csc h(\rho (\frac{V_{n}-\beta }{U(x)})^{\gamma })%
}{e^{\alpha (\frac{V_{n}-\beta }{U(x)})^{\gamma }}}a_{n}\right]
^{q}dx\right\} ^{\frac{1}{q}}>k(\sigma )||a||_{q,\Psi _{\beta }};
\label{5.15}
\end{equation}%
(iii) for $0<p<1,$ $0<||f||_{p,\Phi _{-1}},||a||_{q,\Psi _{\beta }}<\infty ,$
we have the following equivalent inequalities:%
\begin{eqnarray}
&&\sum_{n=1}^{\infty }\int_{0}^{\infty }\frac{\csc h(\rho (\frac{V_{n}-\beta
}{U(x)})^{\gamma })}{e^{\alpha (\frac{V_{n}-\beta }{U(x)})^{\gamma }}}%
a_{n}f(x)dx>k(\sigma )||f||_{p,\widetilde{\Phi }_{-1}}||a||_{q,\Psi _{\beta
}},  \label{5.16} \\
&&\sum_{n=1}^{\infty }\frac{\nu _{n}}{V_{n}^{1-p\sigma }}\left[
\int_{0}^{\infty }\frac{\csc h(\rho (\frac{V_{n}-\beta }{U(x)})^{\gamma })}{%
e^{\alpha (\frac{V_{n}}{U(x)})^{\gamma }}}f(x)dx\right] ^{p}>k(\sigma
)||f||_{p,\widetilde{\Phi }_{-1}},  \label{5.17}
\end{eqnarray}%
\begin{eqnarray}
&&\left\{ \int_{0}^{\infty }\frac{(1-\theta _{-1}(\sigma ,x))^{1-q}\mu (x)}{%
U^{1+q\sigma }(x)}\left[ \sum_{n=1}^{\infty }\frac{\csc h(\rho (\frac{%
V_{n}-\beta }{U(x)})^{\gamma })}{e^{\alpha (\frac{V_{n}-\beta }{U(x)}%
)^{\gamma }}}a_{n}\right] ^{q}dx\right\} ^{\frac{1}{q}}  \notag \\
&>&k(\sigma )||a||_{q,\Psi _{\beta }}.  \label{5.18}
\end{eqnarray}
\end{corollary}

The above inequalities involve the best possible constant factor $k(\sigma
).$

For $\alpha =\rho $ in  Theorem 3.2, Theorem 4.1 and Theorem 4.2, we have

\begin{corollary}
  If $\rho >0,0<\gamma <\sigma \leq 1,$ and $U(\infty )=V(\infty )=\infty ,$ then \\
  (i) for $%
p>1,\ 0<||f||_{p,\Phi _{\delta }},||a||_{q,\Psi _{\beta }}<\infty ,$
we
have the following equivalent inequalities with the best possible constant
factor $$k_{1}(\sigma )=\frac{2\Gamma (\frac{\sigma }{\gamma })\zeta (\frac{%
\sigma }{\gamma })}{\gamma (2\rho )^{\sigma /\gamma }}\::$$%
\begin{eqnarray}
&&\sum_{n=1}^{\infty }\int_{0}^{\infty }\frac{\csc h(\rho U^{\delta \gamma
}(x)(V_{n}-\beta )^{\gamma })}{e^{\rho U^{\delta \gamma }(x)(V_{n}-\beta
)^{\gamma }}}a_{n}f(x)dx<k_{1}(\sigma )||f||_{p,\Phi _{\delta
}}||a||_{q,\Psi _{\beta }},  \label{5.19} \\
&&\sum_{n=1}^{\infty }\frac{\nu _{n+1}}{(V_{n}-\beta )^{1-p\sigma }}\left[
\int_{0}^{\infty }\frac{\csc h(\rho U^{\delta \gamma }(x)(V_{n}-\beta
)^{\gamma })}{e^{\rho U^{\delta \gamma }(x)(V_{n}-\beta )^{\gamma }}}f(x)dx%
\right] ^{p}<k_{1}(\sigma )||f||_{p,\Phi _{\delta }},\ \ \ \label{5.20}
\end{eqnarray}%
\begin{equation}
\left\{ \int_{0}^{\infty }\frac{\mu (x)}{U^{1-q\delta \sigma }(x)}\left[
\sum_{n=1}^{\infty }\frac{\csc h(\rho U^{\delta \gamma }(x)(V_{n}-\beta
)^{\gamma })}{e^{\rho U^{\delta \gamma }(x)(V_{n}-\beta )^{\gamma }}}a_{n}%
\right] ^{q}dx\right\} ^{\frac{1}{q}}<k_{1}(\sigma )||a||_{q,\Psi _{\beta }};
\label{5.21}
\end{equation}%
(ii) for $p<0,0<||f||_{p,\Phi _{\delta }},||a||_{q,\Psi _{\beta }}<\infty ,$
we have the following equivalent inequalities with the best possible
constant factor $k_{1}(\sigma )$:%
\begin{eqnarray}
&&\sum_{n=1}^{\infty }\int_{0}^{\infty }\frac{\csc h(\rho U^{\delta \gamma
}(x)(V_{n}-\beta )^{\gamma })}{e^{\rho U^{\delta \gamma }(x)(V_{n}-\beta
)^{\gamma }}}a_{n}f(x)dx>k_{1}(\sigma )||f||_{p,\Phi _{\delta
}}||a||_{q,\Psi _{\beta }},  \label{5.22} \\
&&\sum_{n=1}^{\infty }\frac{\nu _{n+1}}{(V_{n}-\beta )^{1-p\sigma }}\left[
\int_{0}^{\infty }\frac{\csc h(\rho U^{\delta \gamma }(x)(V_{n}-\beta
)^{\gamma })}{e^{\rho U^{\delta \gamma }(x)(V_{n}-\beta )^{\gamma }}}f(x)dx%
\right] ^{p}>k_{1}(\sigma )||f||_{p,\Phi _{\delta }},  \label{5.23}
\end{eqnarray}%
\begin{equation}
\left\{ \int_{0}^{\infty }\frac{\mu (x)}{U^{1-q\delta \sigma }(x)}\left[
\sum_{n=1}^{\infty }\frac{\csc h(\rho U^{\delta \gamma }(x)(V_{n}-\beta
)^{\gamma })}{e^{\rho U^{\delta \gamma }(x)(V_{n}-\beta )^{\gamma }}}a_{n}%
\right] ^{q}dx\right\} ^{\frac{1}{q}}>k_{1}(\sigma )||a||_{q,\Psi _{\beta }};
\label{5.24}
\end{equation}%
(iii) for $0<p<1,0<||f||_{p,\Phi _{\delta }},||a||_{q,\Psi _{\beta }}<\infty
,$ we have the following equivalent inequalities with the best possible
constant factor $k_{1}(\sigma )$:%
\begin{eqnarray}
&&\sum_{n=1}^{\infty }\int_{0}^{\infty }\frac{\csc h(\rho U^{\delta \gamma
}(x)(V_{n}-\beta )^{\gamma })}{e^{\rho U^{\delta \gamma }(x)(V_{n}-\beta
)^{\gamma }}}a_{n}f(x)dx>k_{1}(\sigma )||f||_{p,\widetilde{\Phi }_{\delta
}}||a||_{q,\Psi _{\beta }},  \label{5.25} \\
&&\sum_{n=1}^{\infty }\frac{\nu _{n+1}}{(V_{n}-\beta )^{1-p\sigma }}\left[
\int_{0}^{\infty }\frac{\csc h(\rho U^{\delta \gamma }(x)(V_{n}-\beta
)^{\gamma })}{e^{\rho U^{\delta \gamma }(x)(V_{n}-\beta )^{\gamma }}}f(x)dx%
\right] ^{p}>k_{1}(\sigma )||f||_{p,\widetilde{\Phi }_{\delta }},
\label{5.26}
\end{eqnarray}%
\begin{eqnarray}
&&\left\{ \int_{0}^{\infty }\frac{(1-\theta _{\delta }(\sigma ,x))^{1-q}\mu
(x)}{U^{1-q\delta \sigma }(x)}\left[ \sum_{n=1}^{\infty }\frac{\csc h(\rho
U^{\delta \gamma }(x)(V_{n}-\beta )^{\gamma })}{e^{\rho U^{\delta \gamma
}(x)(V_{n}-\beta )^{\gamma }}}a_{n}\right] ^{q}dx\right\} ^{\frac{1}{q}}
\notag \\
&>&k_{1}(\sigma )||a||_{q,\Psi _{\beta }}.  \label{5.27}
\end{eqnarray}
\end{corollary}

For $\gamma =\frac{\sigma}{2}$ in Corollary 5.3, we obtain the following:

\begin{corollary}
  If $\rho >0,0<\sigma \leq 1,$ and $U(\infty )=V(\infty )=\infty ,$ then \\
  (i) for $p>1,$ $%
0<||f||_{p,\Phi _{\delta }},||a||_{q,\Psi _{\beta }}<\infty ,$ we have the
following equivalent inequalities with the best possible constant factor $%
\frac{\pi ^{2}}{6\sigma \rho ^{2}}$:%
\begin{equation}
\sum_{n=1}^{\infty }\int_{0}^{\infty }\frac{\csc h(\rho U^{\delta \sigma
/2}(x)(V_{n}-\beta )^{\sigma /2})}{e^{\rho U^{\delta \sigma
/2}(x)(V_{n}-\beta )^{\sigma /2}}}a_{n}f(x)dx<\frac{\pi ^{2}}{6\sigma \rho
^{2}}||f||_{p,\Phi _{\delta }}||a||_{q,\Psi _{\beta }},  \label{5.28}
\end{equation}
\begin{equation}
\sum_{n=1}^{\infty }\frac{\nu _{n+1}}{(V_{n}-\beta )^{1-p\sigma }}\left[
\int_{0}^{\infty }\frac{\csc h(\rho U^{\delta \sigma /2}(x)(V_{n}-\beta
)^{\sigma /2})}{e^{\rho U^{\delta \sigma /2}(x)(V_{n}-\beta )^{\sigma /2}}}%
f(x)dx\right] ^{p}<\frac{\pi ^{2}}{6\sigma \rho ^{2}}||f||_{p,\Phi _{\delta
}},  \label{5.29}
\end{equation}%
\begin{equation}
\left\{ \int_{0}^{\infty }\frac{\mu (x)}{U^{1-q\delta \sigma }(x)}\left[
\sum_{n=1}^{\infty }\frac{\csc h(\rho U^{\delta \sigma /2}(x)(V_{n}-\beta
)^{\sigma /2})}{e^{\rho U^{\delta \sigma /2}(x)(V_{n}-\beta )^{\sigma /2}}}%
a_{n}\right] ^{q}dx\right\} ^{\frac{1}{q}}<\frac{\pi ^{2}}{6\sigma \rho ^{2}}%
||a||_{q,\Psi _{\beta }};  \label{5.30}
\end{equation}%
(ii) for $p<0,0<||f||_{p,\Phi _{\delta }},||a||_{q,\Psi _{\beta }}<\infty ,$
we have the following equivalent inequalities with the best possible
constant factor $\frac{\pi ^{2}}{6\sigma \rho ^{2}}$:%
\begin{equation}
\sum_{n=1}^{\infty }\int_{0}^{\infty }\frac{\csc h(\rho U^{\delta \sigma
/2}(x)(V_{n}-\beta )^{\sigma /2})}{e^{\rho U^{\delta \sigma
/2}(x)(V_{n}-\beta )^{\sigma /2}}}a_{n}f(x)dx>\frac{\pi ^{2}}{6\sigma \rho
^{2}}||f||_{p,\Phi _{\delta }}||a||_{q,\Psi _{\beta }},  \label{5.31}
\end{equation}
\begin{equation}
\sum_{n=1}^{\infty }\frac{\nu _{n+1}}{(V_{n}-\beta )^{1-p\sigma }}\left[
\int_{0}^{\infty }\frac{\csc h(\rho U^{\delta \sigma /2}(x)(V_{n}-\beta
)^{\sigma /2})}{e^{\rho U^{\delta \sigma /2}(x)(V_{n}-\beta )^{\sigma /2}}}%
f(x)dx\right] ^{p}>\frac{\pi ^{2}}{6\sigma \rho ^{2}}||f||_{p,\Phi _{\delta
}},  \label{5.32}
\end{equation}%
\begin{equation}
\left\{ \int_{0}^{\infty }\frac{\mu (x)}{U^{1-q\delta \sigma }(x)}\left[
\sum_{n=1}^{\infty }\frac{\csc h(\rho U^{\delta \sigma /2}(x)(V_{n}-\beta
)^{\sigma /2})}{e^{\rho U^{\delta \sigma /2}(x)(V_{n}-\beta )^{\sigma /2}}}%
a_{n}\right] ^{q}dx\right\} ^{\frac{1}{q}}>\frac{\pi ^{2}}{6\sigma \rho ^{2}}%
||a||_{q,\Psi _{\beta }};  \label{5.33}
\end{equation}%
(iii) for $0<p<1,0<||f||_{p,\Phi _{\delta }},||a||_{q,\Psi _{\beta }}<\infty
,$ we have the following equivalent inequalities with the best possible
constant factor $\frac{\pi ^{2}}{6\sigma \rho ^{2}}$:%
\begin{equation}
\sum_{n=1}^{\infty }\int_{0}^{\infty }\frac{\csc h(\rho U^{\delta \sigma
/2}(x)(V_{n}-\beta )^{\sigma /2})}{e^{\rho U^{\delta \sigma
/2}(x)(V_{n}-\beta )^{\sigma /2}}}a_{n}f(x)dx>\frac{\pi ^{2}}{6\sigma \rho
^{2}}||f||_{p,\widetilde{\Phi }_{\delta }}||a||_{q,\Psi _{\beta }},
\label{5.34}
\end{equation}
\begin{equation}
\sum_{n=1}^{\infty }\frac{\nu _{n+1}}{(V_{n}-\beta )^{1-p\sigma }}\left[
\int_{0}^{\infty }\frac{\csc h(\rho U^{\delta \sigma /2}(x)(V_{n}-\beta
)^{\sigma /2})}{e^{\rho U^{\delta \sigma /2}(x)(V_{n}-\beta )^{\sigma /2}}}%
f(x)dx\right] ^{p}>\frac{\pi ^{2}}{6\sigma \rho ^{2}}||f||_{p,\widetilde{%
\Phi }_{\delta }},  \label{5.35}
\end{equation}%
\begin{eqnarray}
&&\left\{ \int_{0}^{\infty }\frac{(1-\theta _{\delta }(\sigma ,x))^{1-q}\mu
(x)}{U^{1-q\delta \sigma }(x)}\left[ \sum_{n=1}^{\infty }\frac{\csc h(\rho
U^{\delta \sigma /2}(x)(V_{n}-\beta )^{\sigma /2})}{e^{\rho U^{\delta \sigma
/2}(x)(V_{n}-\beta )^{\sigma /2}}}a_{n}\right] ^{q}dx\right\} ^{\frac{1}{q}}
\notag \\
&>&\frac{\pi ^{2}}{6\sigma \rho ^{2}}||a||_{q,\Psi _{\beta }}.  \label{5.36}
\end{eqnarray}%
\textbf{\ }
\end{corollary}

\begin{remark}$ $\\
  (i) For $\beta =0$ in (\ref{3.1}), the
following inequality holds true:%
\begin{equation}
\sum_{n=1}^{\infty }\int_{0}^{\infty }\frac{\csc h(\rho (U^{\delta
}(x)V_{n})^{\gamma })}{e^{\alpha (U^{\delta }(x)V_{n})^{\gamma }}}%
a_{n}f(x)dx<k(\sigma )||f||_{p,\Phi _{\delta }}||a||_{q,\Psi _{0}}.
\label{5.37}
\end{equation}%
Hence, (\ref{3.1}) is a more accurate inequality of (\ref{5.37}) for $%
0<\beta \leq \frac{\nu _{1}}{2}.$
\end{remark}

\noindent (ii) For $\mu (x)=\nu _{n}=1$ in (\ref{5.37}), we have the following
inequality with the best possible constant factor $k(\sigma ):$%
\begin{eqnarray}
&&\sum_{n=1}^{\infty }\int_{0}^{\infty }\frac{\csc h(\rho (x^{\delta
}n)^{\gamma })}{e^{\alpha (x^{\delta }n)^{\gamma }}}a_{n}f(x)dx  \notag \\
&<&k(\sigma )\left[ \int_{0}^{\infty }x^{p(1-\delta \sigma )-1}f^{p}(x)dx%
\right] ^{\frac{1}{p}}\left[ \sum_{n=1}^{\infty }n^{q(1-\sigma )-1}a_{n}^{q}%
\right] ^{\frac{1}{q}}.  \label{5.38}
\end{eqnarray}

In particular, for $\delta =1,$ we have the following inequality with the
non-homogeneous kernel:%
\begin{eqnarray}
&&\sum_{n=1}^{\infty }\int_{0}^{\infty }\frac{\csc h(\rho (xn)^{\gamma })}{%
e^{\alpha (xn)^{\gamma }}}a_{n}f(x)dx  \notag \\
&<&k(\sigma )\left[ \int_{0}^{\infty }x^{p(1-\sigma )-1}f^{p}(x)dx\right] ^{%
\frac{1}{p}}\left[ \sum_{n=1}^{\infty }n^{q(1-\sigma )-1}a_{n}^{q}\right] ^{%
\frac{1}{q}};  \label{5.39}
\end{eqnarray}%
for $\delta =-1,$ we have the following inequality with the homogeneous
kernel of degree 0:%
\begin{eqnarray}
&&\sum_{n=1}^{\infty }\int_{0}^{\infty }\frac{\csc h(\rho (\frac{n}{x}%
)^{\gamma })}{e^{\alpha (\frac{n}{x})^{\gamma }}}a_{n}f(x)dx  \notag \\
&<&k(\sigma )\left[ \int_{0}^{\infty }x^{p(1+\sigma )-1}f^{p}(x)dx\right] ^{%
\frac{1}{p}}\left[ \sum_{n=1}^{\infty }n^{q(1-\sigma )-1}a_{n}^{q}\right] ^{%
\frac{1}{q}}.  \label{5.40}
\end{eqnarray}
$$ $$
%\begin{Acknowledgements}
\noindent\textbf{Acknowledgements}

The authors wish to express their thanks to the referees for their careful reading of the manuscript and for their valuable suggestions.\\
We would like to thank Professors J. C. Kuang and M. Krni\'c for their very useful comments.\\
B. Yang: This work is supported by the National Natural Science Foundation of China (No. 61370186), and 2013
Knowledge Construction Special Foundation Item of Guangdong Institution of
Higher Learning College and University (No. 2013KJCX0140). We are grateful for their help.\\
M. Th. Rassias: This work is supported by the
SNF grant: SNF PP00P2\textunderscore138906. I would like to express my gratitude to Professor P. -O. Dehaye and the Swiss National Science Foundation for providing me with financial support to conduct postdoctoral research at the University of Zurich during the academic year 2015-2016.
%\end{Acknowledgements}

\end{document}